\documentclass[authorversion]{acmart}
\usepackage[utf8]{inputenc}
\renewcommand\footnotetextcopyrightpermission[1]{} % removes footnote with conference information in first column
\usepackage{amsmath,amssymb,amsthm,dsfont}

\usepackage{todonotes}
\usepackage{subcaption}
\usepackage{chemformula}

\usepackage{tikz}
\usetikzlibrary{arrows,shapes.gates.logic.US,shapes.gates.logic.IEC,calc}
\usetikzlibrary{decorations.pathreplacing}
\usetikzlibrary{spy}

\usepackage{multicol}

\newcommand{\IN}{\mathds{N}}
\newcommand{\IZ}{\mathds{Z}}
\newcommand{\IE}{\mathds{E}}
\newcommand{\IP}{\mathds{P}}
\newcommand{\IR}{\mathds{R}}

\newcommand{\Bc}{\textbf{c}}
\newcommand{\Br}{\textbf{r}}
\newcommand{\Bp}{\textbf{p}}
\newcommand{\CS}{\mathcal{S}}

\DeclareMathOperator{\Var}{Var}

\newcommand{\nop}[1]{}

\newcommand{\NAND}{\textsc{Nand}}
\newcommand{\figref}[1]{Fig.~\ref{#1}}

\usepackage{mathtools}  
\mathtoolsset{showonlyrefs}

\usepackage[appendix=inline]{apxproof}
\usepackage[inline]{optional}

\newtheoremrep{thm}{Theorem}
\newtheoremrep{cor}[thm]{Corollary}
\newtheoremrep{lem}[thm]{Lemma}

\AtBeginDocument{%
  \providecommand\BibTeX{{%
    \normalfont B\kern-0.5em{\scshape i\kern-0.25em b}\kern-0.8em\TeX}}}

\begin{document}
%% Use this when submitting an article to a sponsored event. You'll
%% receive a unique submission ID from the organizers
%% of the event, and this ID should be used as the parameter to this command.
%%\acmSubmissionID{123-A56-BU3}

%%
%% The majority of ACM publications use numbered citations and
%% references.  The command \citestyle{authoryear} switches to the
%% If you are preparing content for an event
%% sponsored by ACM SIGGRAPH, you must use the "author year" style of
%% the next command will enable that style.

%%
%% end of the preamble, start of the body of the document source.

\title{Distributed Computation with Continual Population Growth}


\author{Da-Jung Cho}
\affiliation{%
  \institution{University of Kassel}
  \city{Kassel}
  \country{Germany}}
\email{dajung.cho@uni-kassel.de}

\author{Matthias F\"ugger}
\affiliation{%
  \institution{CNRS, LSV, ENS Paris-Saclay, Universit\'e Paris-Saclay, Inria}
  \city{Cachan}
  \country{France}}
\email{mfuegger@lsv.fr}

\author{Corbin Hopper}
\affiliation{%
  \institution{ENS Paris-Saclay}
  \city{Cachan}
  \country{France}}
\affiliation{%
  \institution{Universit\'e Paris-Saclay, CNRS}
  \city{Orsay}
  \country{France}}
\email{corbin.hopper@mail.mcgill.ca}

\author{Manish Kushwaha}
\affiliation{%
  \institution{Universit\'e Paris-Saclay, INRAE, AgroParisTech, Micalis Institute}
  \city{Jouy-en-Josas}
  \country{France}}
\email{manish.kushwaha@inrae.fr}

\author{Thomas Nowak}
\affiliation{%
  \institution{Universit\'e Paris-Saclay, CNRS}
  \city{Orsay}
  \country{France}}
\email{thomas.nowak@lri.fr}

\author{Quentin Soubeyran}
\affiliation{%
  \institution{\'Ecole polytechnique}
  \city{Palaiseau}
  \country{France}}
\affiliation{%
  \institution{Universit\'e Paris-Saclay, CNRS}
  \city{Orsay}
  \country{France}}
\email{quentin.soubeyran@polytechnique.edu}

\begin{abstract}
Computing with synthetically engineered bacteria is a vibrant and active field with numerous applications in bio-production, bio-sensing, and medicine. Motivated by the lack of robustness and by resource limitation inside single cells, distributed approaches with communication among bacteria have recently gained in interest. In this paper, we focus on the problem of population growth happening concurrently, and possibly interfering, with the desired bio-computation. Specifically, we present a fast protocol in systems with continuous population growth for the majority consensus problem and prove that it correctly identifies the initial majority among two inputs with high probability if the initial difference is $\Omega(\sqrt{n\log n})$ where~$n$ is the total initial population.
We also present a fast protocol that correctly computes the NAND of two inputs with high probability.
We demonstrate that combining the NAND gate protocol with the continuous-growth majority consensus protocol, using the latter as an amplifier, it is possible to implement circuits computing arbitrary Boolean functions.
\end{abstract}




% %% A "teaser" image appears between the author and affiliation
% %% information and the body of the document, and typically spans the

\maketitle

\section{Introduction}

In the past few decades, synthetic biology has laid considerable focus on the re-programming of cells as computing machines. They have been engineered to sense a range of inputs (metabolites~\cite{slomovic2015synthetic}, light~\cite{tabor2009synthetic}, oxygen~\cite{anderson2006environmentally}, pH~\cite{schmidl2019rewiring}) and process them to produce desired outputs according to defined processing codes (primarily digital~\cite{moon2012genetic}, but occasionally analog~\cite{daniel2013synthetic}). Some potential applications of the cellular machines include production of metabolic compounds of interest~\cite{paddon2013high}, bio-remediation of toxic environments~\cite{tay2017synthetic}, sensing of disease bio-markers~\cite{slomovic2015synthetic}, and therapeutic intervention by targeted effector delivery~\cite{anderson2006environmentally}. Yet, the ability of single cells to reliably process multiple inputs is acutely constrained by their limited resources. 

Adding too many processes into one cell leads to resource-stress and eventually the code is lost due to mutation, a baseline error mechanism present in all living systems. This has, in part, encouraged the notion of distributing the computational tasks across multiple cells~\cite{regot2011distributed,tamsir2011robust}, to reduce resource-stress and improve robustness. The value of the idea is corroborated by the success of the division of labor seen in multi-cellular organisms that have naturally evolved from their unicellular ancestors~\cite{libby2014ratcheting,ratcliff2012experimental}. While task-distribution in cell populations solves some problems, it immediately leads to other challenges that must be tackled for the successful implementation of any complex distributed program. Some of these challenges include: the orthogonality/specificity of communication signals, the rate and bandwidth of communication channels, cellular growth and its effect on signal amplification or dissipation, and effect of cross-talk between different signals.

In this work we focus on amplification and Boolean function computation in distributed systems whose agents are duplicating bacteria.
A central problem in this setting is to maintain a consistent state of circuit values among the bacteria, a problem that has been studied in distributed computing for decades in different contexts \cite{Lynch96}.
Starting from a mathematical computing model, analysis of a system's behavior has led to correctness proofs and performance bounds of proposed solutions, also shedding light on how protocol parameters influence the quality of the outcome.
In distributed systems with biological agents, the cellular behavior is usually expressed in the language of chemical reaction networks~(CRNs).
A CRN is defined by a set of reactions, each consuming members of one or several species and producing members of others at a given rate.

The two most commonly used kinetics for CRNs are deterministic and stochastic approaches. The deterministic approach models the kinetics of a CRN as systems of ordinary differential equations (ODEs) with continuous real-valued concentrations of each species, whereas the stochastic approach models the CRN as a continuous-time Markov chain with discrete integer-valued counts of each species.  
While ODE modeling can capture important behavioral characteristics, in particular expected-value large-population limits, some phenomena can only be explained by stochastic-process kinetics.
In particular, ODE kinetics cannot elucidate the probability of certain population-level events occurring in a system of two competing species, e.g., the extinction of one species due to a series of random events.
The stochastic-process kinetics of CRNs are much more common in distributed computing, in particular in population protocols \cite{angluin08:dc}, where reactions are restricted to two reactants and two products with constant size populations, but also in computability results in more general CRNs~\cite{DBLP:journals/nc/SoloveichikCWB08}.

\paragraph{Consistent cell states by competition among cells.}
Competition among species naturally lends itself to solving consensus-type problems.
Angluin \emph{et al.} \cite{angluin08:dc} analyzed a population protocol with three states: $A$, $B$, and blank.
Encounters of opposing species $A$ and $B$ lead to one of them
  becoming blank, and blank species that encounter a non-blank species copy its state.
The population protocol by Angluin \emph{et al.} \cite{angluin2008fast} alternates phases of state duplication and cancellation, separated by a clock signal generated by a dedicated leader species.
These protocols, however, rely on constant size populations and the latter on a dedicated leader, rendering them impractical for implementations in bacterial cultures.

Birth-death processes track species counts within a population with ``birth'' and ``death'' events over time.
For each such population state there are transitions that move from one population state to the other with respect to ``birth'' and ``death'' events.   
Birth-death processes have been used to model competition, predation, or infection in evolutionary biology, ecology, genetics, and queueing theory~\cite{NKK06,Saaty61}. 

An early mention of problems requiring a stochastic analysis of two competing species is by
  Volterra~\cite{volterraleccons} and Feller~\cite{feller39grundlagen}
although only the growth of a single species is analyzed therein.
For an overview of single species birth-death Markov chains, see, e.g., \cite{bremaud99}.
Extensions for multiples species, with applications to genetic mutations, are found in
  the literature on competition and branching processes~\cite{reuter1961competition,billard1974competition,kendall1966branching}. For example, Ridler-Rowe \cite{ridler1978competition} considers a stochastic process between two competing species.
However, the process in that work differs from ours in that death reactions
  are \ch{$A$ + $B$ -> $A$} and \ch{$A$ + $B$ -> $B$}, leaving
  a winner after an encounter between two competing individuals.
The paper presents an approximation for long-term distributions and bounds the probability that starting from initial $A,B$ sizes, species $A$ goes extinct.
However, the analysis is for initial population sizes approaching infinity, only, and assumes an initial gap between species counts that is linear in the population size.
By contrast our analysis holds for finite population sizes $n$, and requires a gap of $\Omega(\sqrt{n \log n})$, only.
A complementary approach for the same asymmetric process proposed in \cite{gomez2012extinction} is to numerically solve
  a finite size cut-off of the infinite linear equation systems.

\paragraph{Computation in birth systems.}
In this work, we introduce and study \emph{protocols for birth systems}
  where all species inherently duplicate.
Such protocols are thus different from population protocols, which have population sizes that remain constant over the course of an execution.
Further, our protocols do not rely on exact species counts, they are not leader-based, and they require small and constant state space per cell, lending themselves readily for future biological implementation.

For simplicity we assume that all duplication reactions of our birth systems have the same rate.
We leave the question of natural selection due to differing growth rates to future work.
In particular, we study two protocols within birth systems.
\begin{itemize}
\item[(i)] We introduce the A-B protocol for two species $A$ and $B$
  and show that it solves majority consensus with high probability: If the initial difference between $A$ and $B$ sizes $\Delta$ grows weakly with the population size $n$ according to~$\Delta=\Omega(\sqrt{n \log n})$, then the protocol identifies the initial majority with high probability.
Since it amplifies the difference between the two species, we also refer to the A-B protocol as an \emph{amplifier}.
Further, we will show that the protocol reaches consensus in expected constant time.
The protocol's reactions are deceptively simple.
Besides the obligatory birth reactions~\ch{$A$ -> $2A$} and \ch{$B$ -> $2B$}, it comprises a single death reaction~\ch{$A$ + $B$ -> $\emptyset$}. 

\item[(ii)] We demonstrate how to implement the components of \emph{feed-forward Boolean circuits}.
Each Boolean gate in our implementation is a \NAND\ gate, followed by an amplifier.
Note that while we focus on the universal \NAND\ gate for the sake of a lighter notation, our construction and its analysis holds for any arbitrary two-input Boolean function.
The latter will be important for optimization and follow-up with biological implementations.
Signals between the \NAND\ gates
  are encoded using two species each, the difference of
  which determines whether a signal is a logical $0$, $1$, or neither.
A \NAND\ gate is a protocol that maps two input
  signals $X$ and $Y$ to an output signal~$Z$
  that is the logical \NAND\ of $X$ and $Y$.
\end{itemize}

While \NAND\ gates are used to implement the circuit's
  Boolean behavior, the successive amplifiers regenerate the gate's output signal by amplifying the difference between the two output signal species.
Repeated, successive invocation of the \NAND\ protocol followed by the amplifier protocol for time $O(\log n)$, where $n$ is the total initial population, can finally be used to compute the
  circuit's output values layer by layer.

\paragraph{Organization.} The rest of the paper is organized as follows:
In Section~\ref{sec:model}, we define the computational model. In Section~\ref{sec:majority_consensus}, we introduce and analyze our protocol for majority consensus.
In Section~\ref{sec:boolean_gate}, we define and analyze the \NAND\ gate protocol. In Section~\ref{sec:simulations}, we present simulations of the A-B protocol as well as a biologically plausible implementation of the \NAND\ gate with amplifiers. Finally, Section~\ref{sec:conclusion} concludes the paper by summarizing our results.

\section{Model}\label{sec:model}

We write $\IN = \{0,1,\dots\}$, $\IN^+ = \IN \setminus \{0\}$, and $\IR_0^+ = [0,\infty)$.
When analyzing our protocols, we employ the term ``with high probability'' relative to the total initial population.
That is, event~$E$ happens with high probability if there exists some $c>0$ such that $\IP(E) = 1-O\left(1/n^c\right)$, where~$n$ is the total initial population.

\subsection{Chemical Reaction Networks}

We use the standard stochastic kinetics for chemical reaction networks.
A reader familiar with the model can safely skip this subsection.

A \emph{chemical reaction network} is described by a set~$\CS$ of species and a set of reactions.
A \emph{reaction} is a triple~$(\Br, \Bp, \alpha)$ where $\Br, \Bp\in \mathbb{N}^{\CS}$ and $\alpha\in\IR_0^+$.
The species with positive count in $\Br$ are called the reaction's \emph{reactants} and this with positive count in~$\Bp$ are called its \emph{products}.
The parameter $\alpha$ is called the reaction's \emph{rate constant}.
A \emph{configuration} of a CRN is simply an element of $\IN^{\CS}$.
A reaction~$(\Br, \Bp, \alpha)$ is \emph{applicable} to configuration~$\Bc$ if $\Br(S) \leq \Bc(S)$ for all $S\in\CS$.

We will write \ch{$\Br$ ->[$\alpha$] $\Bp$} to denote a reaction $(\Br, \Bp, \alpha)$.
For instance, the reaction~$(\{A, B\},\{2B, C\}, \alpha)$ will simply be denoted 
\ch{$A+B$ ->[$\alpha$] $2B+C$}.
Here, we used the shorthand notations $\{A, B\}$ and $\{2B,C\}$ for functions $\CS\to\IN$.
For instance, the notation $\{2B,C\}$ represents the function $\Bp:\CS\to\IN$ defined by
$\Bp(B)=2$, $\Bp(C)=1$, and
$\Bp(S)=0$ for all other species $S \not\in \{B,C\}$.

The \emph{stochastic kinetics} of a CRN are a continuous-time Markov chain (see a textbook~\cite{bremaud99} for auxiliary definitions).
Given some volume $v\in\IR_0^+$, which we will normalize to $v=1$, the propensity of a reaction~$(\Br, \Bp, \alpha)$ in configuration~$\Bc$ is equal to
$\frac{\alpha}{v}
\prod_{S\in\CS} \binom{\Bc(S)}{\Br(S)}$, 
where $\binom{\Bc(S)}{\Br(S)}$ denotes the binomial coefficient of~$\Bc(S)$ and~$\Br(S)$.
The binomial coefficient is~$1$ if $\Br(S)=0$, i.e., if the species~$S$ is not a reactant of the reaction.
It is~$0$ if $\Br(S) > \Bc(S)$.
The propensity of a non-applicable reaction is thus~$0$.
For example, the propensity of reaction \ch{$A+B$ ->[$\alpha$] $2B+C$} in configuration~$\Bc$ is equal to $\frac{\alpha}{v}\cdot \Bc(A) \cdot \Bc(B)$.
The propensity of \ch{$A$ ->[$\gamma$] $2A$} is equal to $\frac{\gamma}{v}\cdot \Bc(A)$.
The new configuration after an applicable reaction is equal to 
$\Bc' = \Bc - \Br + \Bp$.

We will use the notation~$Q(x,y)$ for
the propensity of the transition from state~$x$ to state~$y$
in a continuous-time Markov chain.
To each continuous-time Markov chain corresponds a discrete-time Markov chain that only keeps track of the sequence of state changes, but not of their timing.
We will write~$P(x,y)$ for the transition probability from state~$x$ to state~$y$ in the discrete-time chain.
We have the formula
$P(x,y) = Q(x,y)/\sum_z Q(x,z)$.

\subsection{Birth Systems}

A \emph{protocol for a birth system}, or protocol, with input species $\mathcal{I}$ and
  output species $\mathcal{O}$, for finite, not necessarily disjoint, sets $\mathcal{I}$ and $\mathcal{O}$
  is a CRN specified as follows.
Its set of species $\mathcal{S}$ comprises input/output
  species $\mathcal{I} \cup \mathcal{O}$ and a finite set of internal species $\mathcal{L}$.
Further, the protocol defines the initial
  species counts $X_0$ for internal and output species $X \in \mathcal{L} \cup \mathcal{O}$ and a
  finite set of reactions $\mathcal{R}$ on the species in $\mathcal{S}$.
For each species $X \in \mathcal{S}$, there is a duplication reaction of the form
  $\ch{$X$ ->[$\gamma$] $2X$}$.
All duplication reactions have the same rate constant~$\gamma>0$.

Given a protocol and an initial species count for its inputs,
an execution of the protocol is given by the stochastic process of the CRN with species $\mathcal{S}$,
reactions $\mathcal{R}$, and respective initial species counts.

\section{Majority Consensus}
\label{sec:majority_consensus}
The A-B protocol is defined for two species, $A$ and $B$, both of which are inputs and outputs.
It contains, apart from the obligatory duplication reactions, the single reaction of~$A$ and~$B$ eliminating each other with rate constant $\delta>0$.
The complete list of reactions of the A-B protocol is thus:
\begin{equation*}
\begin{split}
\ch{$A$ ->[$\gamma$] $2A$}\quad\quad
\ch{$B$ ->[$\gamma$] $2B$}\quad\quad
\ch{$A+B$ ->[$\delta$] $\emptyset$}
\end{split}
\end{equation*}

We say that \emph{consensus} is reached if one of the two species becomes extinct.
If the initial population counts differ, we say that \emph{majority consensus} is reached if consensus is reached and the species that was initially in majority is not extinct.
If the initial counts of both species are equal, then majority consensus is reached when one species is extinct and the other is not.

We show that the A-B protocol reaches consensus in constant time and majority consensus with high probability.

\begin{thm}\label{thm:consensus}
For initial population~$n = A(0) + B(0)$ and initial gap~$\Delta = \lvert A(0)-B(0)\rvert$,
the A-B protocol reaches consensus in expected time  $O(1)$ and in time  $O(\log n)$ with high probability.
It reaches majority consensus with probability\/ 
$1 - e^{-\Omega(\Delta^2/n)}$.
\end{thm}

From Theorem~\ref{thm:consensus} we immediately obtain a bound on the initial gap sufficient for majority consensus with high probability.

\begin{cor}
For initial population~$n$ and initial gap~$\Delta$, if $\Delta = \Omega\left(\sqrt{n \log n}\right)$, then the A-B protocol reaches majority consensus with high probability.
\end{cor}

Without duplication reactions, it is obvious that the A-B protocol reaches consensus and that majority consensus is always reached if the two species have different initial population counts.
We are thus not only able to show that we can achieve majority consensus in spite of continual population growth via duplication reactions of all species, but also that a sub-linear gap in the initial population counts suffices.
The required initial gap of $\Omega(\sqrt{n\log n})$ matches that of the best protocols without obligatory duplications~\cite{angluin08:dc,condon19majority}.

We will prove Theorem~\ref{thm:consensus} in the following sections; first the time upper bound, then correctness with high probability.

\subsection{Markov-Chain Model}

The A-B protocol evolves as a continuous-time Markov chain with state space $S = \IN^2$.
Its state-transition rates are:
\begin{equation*}
\begin{split}
Q\big( (A,B) \,,\, (A+1,B) \big) & = \gamma A\\
Q\big( (A,B) \,,\, (A,B+1) \big) & = \gamma B\\
Q\big( (A,B) \,,\, (A-1,B-1) \big) & = \delta A B
\end{split}
\end{equation*}
Note that the death transition $(A,B) \to (A-1,B-1)$ has rate zero if $A=0$ or $B=0$.
Both axes $\{0\}\times\IN$ and $\IN\times\{0\}$ are absorbing, and so is the state  $(A,B)=(0,0)$.
This chain is regular, i.e., its sequence of transition times is unbounded with probability~$1$.
Indeed, as we will show,  the discrete-time chain reaches consensus with probability~$1$, from which time on the chain is equal to a linear pure-birth process, which is regular.

The corresponding discrete-time jump chain has the same state space $S=\IN^2$ and the state-transition probabilities
\begin{equation}
\begin{split}
P\big( (A,B) \,,\, (A+1,B) \big) & = \frac{\gamma A}{\gamma(A+B) + \delta AB}\\
P\big( (A,B) \,,\, (A,B+1) \big) & = \frac{\gamma B}{\gamma(A+B) + \delta AB}\\
P\big( (A,B) \,,\, (A-1,B-1) \big) & = \frac{\delta AB}{\gamma(A+B) + \delta AB}
\end{split}
\end{equation}
if $A>0$ or $B>0$.
The axes as well as state $(A,B)=(0,0)$ is absorbing, as in the continuous-time chain.

As a convention, we will write~$X(t)$ for the state of the continuous-time process~$X$ at time~$t$, and~$X_k$ for the state of the discrete-time jump process after~$k$ state transitions.
The time to reach consensus is the earliest time~$T$ such that $A(T)=0$ or $B(T)=0$.

\subsection{Time to Reach Consensus}

In this section we prove the first part of Theorem~\ref{thm:consensus}, i.e., the bounds on the time to reach consensus, both in expected time and with high probability.
For that, we will employ a coupling of the A-B protocol Markov chain with a single-species birth-death process.
We show that the A-B protocol reaches consensus when the single-species process reaches its extinction state and then bound this time in the single-species process.
\figref{fig:coupling:ineq:principle} visualizes the idea of the proof. 

We denote the single-species process by~$M(t)$.
It is a birth-death chain with state space $S=\IN$ and transition rates $Q(M,M+1) = \gamma M$ and $Q(M,M-1) = \delta M^2$.
State~$0$ is absorbing.
Note that the death rate~$\delta M^2$ depends quadratically on the current population~$M$, and not linearly like the birth rate~$\gamma M$.
The reason is that we want~$M(t)$ to bound the minimum of the populations~$A(t)$ and~$B(t)$ and that the death transition in the A-B protocol is quadratic in this minimum.

We will crucially use the fact that $\IP\bigr(M(t)=0\bigl) \leq \IP\bigr(A(t)=0 \vee B(t)=0\bigl)$ for all times~$t$.
This, together with a bound on the  time until $M(t)=0$, then gives a bound on the  time until consensus in the A-B protocol chain.

\begin{figure}[h]
\centering
\begin{tikzpicture}[spy using outlines={circle,magnification=4, size=4cm, connect spies}]
\coordinate (N) at (0,0);
\coordinate (T) at (8,0);
\coordinate (C) at (0,3);
\coordinate (Init) at (0,2);

\def\mdt{0.2}
\def\mdc{0.2}
\newcommand{\upafter}[1]{ -- ++($(#1 * \mdt,0)$) -- ++(0,\mdc)}
\newcommand{\downafter}[1]{ -- ++($(#1 * \mdt,0)$) -- ++(0,-\mdc)}
\newcommand{\straight}[1]{ -- ++($(#1 * \mdt,0)$)}

\draw[->] (N) -- (T) node[right] {$t$} node[pos=0,below] {$0$};
\draw[->] (N) -- (C) node[above] {};

\draw[green!50!black,thick] (Init)
  \upafter{2}
  \downafter{1}
  \upafter{1}\upafter{3}
  \downafter{1}\downafter{1}\downafter{3}\downafter{2}
  \downafter{2}\downafter{1}
  \upafter{3}
  \downafter{1}\downafter{2}\downafter{1}\downafter{3}
  \downafter{1}\downafter{2}\downafter{1}
  node[pos=1,below,xshift=-8pt] {$\min\{A(t),B(t)\}$};
  
\draw[orange!80!black,thick] (Init)
  \upafter{1}
  \upafter{0.5}
  \upafter{1.5}\upafter{2}
  \downafter{0.5}\downafter{2}\downafter{3}\downafter{2}
  \downafter{2}\upafter{3}
  \upafter{1}
  \downafter{3}\downafter{2}\downafter{1}\downafter{1}
  \downafter{1}\downafter{2}\downafter{3}\downafter{1}
  \downafter{2}\downafter{1}\downafter{1}
  node[pos=1,below,xshift=8pt] {$M(t)$};
  
\node[xshift=-30pt,yshift=7pt,orange!80!black,thick] at (Init) {$M(0) =$}; \node[xshift=-30pt,yshift=-7pt,green!50!black,thick] at (Init) {$\min\{A(0),B(0)\}$};

\draw[<->] (4.1,1.45) -- ++(0,0.7) node[above,yshift=5pt] {Lemma~\ref{lem:coupling:ineq}};

\draw[<->] (0,-0.55) -- ++(7.29,0) node[midway,below] {Lemma~\ref{lem:min:expected:time}};

% ends
\draw[fill,green!50!black] (Init) circle (0.7mm);
\draw[fill,green!50!black] (6.21,0) circle (0.7mm);
\draw[fill,orange!80!black] (7.29,0) circle (0.7mm);

% zoom
\spy[black] on (6.1,1.1)
    in node [left] at (13,1.1);
\path[->,>=stealth',thick,blue!40!white] (9.5,0.85) edge[bend left=1.8cm] ++(0.45,0);
\path[->,>=stealth',thick,blue!40!white] (10,0.85) edge[bend left=1.5cm] ++(0.95,0);
\path[->,>=stealth',thick,orange!80!black] (11,0.85) edge[bend left=1.5cm] ++(0.8,0);
% long arrow
\path[->,>=stealth',dashed,thick] (9.5,1) edge[bend left=2cm] node[above,yshift=4pt] {Lemma~\ref{lem:markov:stuttering}} ++(2.3,0);

\end{tikzpicture}
\caption{Idea of the proof: Construction of a continuous-time coupling of the A-B protocol and the single species birth-death M chain.
Stuttering steps are mapped to effective steps (Lemma~\ref{lem:markov:stuttering}).
An execution of the coupling process fulfills the deterministic guarantee $\min\{A(t),B(t)\}\leq M(t)$ for all times $t\geq0$ (Lemma~\ref{lem:coupling:ineq}).
From the coupling it follows that
$\IP\bigr(M(t)=0\bigl) \leq \IP\bigr(A(t)=0 \vee B(t)=0\bigl)$ for the uncoupled processes (Lemma~\ref{lem:coupling:probs}).
The time until consensus then follows from the time until extinction in the M chain (Lemma~\ref{lem:min:expected:time}).
}
\label{fig:coupling:ineq:principle}
\end{figure}
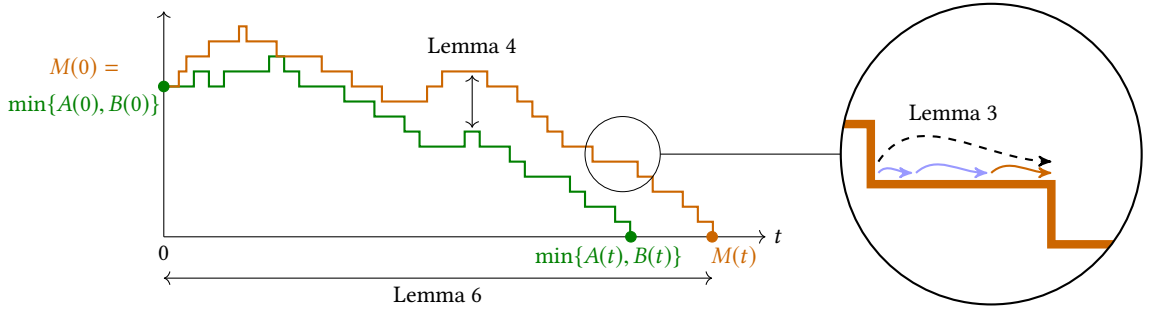

\paragraph{Continuous-time coupling.}
The coupling is defined as follows.
For sequences~$(\xi_k)_{k\geq1}$ of i.i.d.\ (independent and identically distributed) uniform random variables in the unit interval $[0,1)$ and $(\eta_k)_{k\geq1}$ of i.i.d.\ exponential random variables with normalized rate~$1$, we define the coupled process $(A(t),B(t),M(t))$ as follows.
Initially, $M(0) = \min\{A(0),B(0)\}$.
For $k \geq 0$, the $(k+1)$\textsuperscript{th} transition happens after time $\eta_k/\Lambda(A_k,B_k,M_k)$ where
$\Lambda(A,B,M) = \max\{\lambda(A,B),\lambda(M)\}$
is the maximum of the sums of transition rates of the individual chains in states $(A,B)$ and~$M$, respectively, i.e.,~$\lambda(A,B)=\gamma(A+B)+\delta AB$
and
$\lambda(M) = \gamma M + \delta M^2$.
The new state $(A_{k+1},B_{k+1},M_{k+1})$ of the coupled chain is then determined by the following update rules.
The state $(0,0,0)$ is absorbing.
Otherwise, if $A_k \leq B_k$, then:
\begin{equation}\label{eq:min:coupling:ab}
(A_{k+1},B_{k+1})
=
\begin{cases}
(A_k+1,B_k) & \text{if } \xi_{k+1} \in \left[ 0 \,,\, \frac{\gamma A_k}{\Lambda(A_k,B_k,M_k)} \right)\\
(A_k,B_k+1) & \text{if } \xi_{k+1} \in \left[  \frac{\gamma A_k}{\Lambda(A_k,B_k,M_k)} \,,\, \frac{\gamma A_k+\gamma B_k}{\Lambda(A_k,B_k,M_k)} \right)\\
(A_k-1,B_k-1) & \text{if } \xi_{k+1} \in \left[ 1 - \frac{\delta A_k B_k}{\Lambda(A_k,B_k,M_k)} \,,\, 1 \right)\\
(A_k,B_k) & \text{otherwise}
\end{cases}
\end{equation}
If $A_k > B_k$ then the roles of~$A_k$ and~$B_k$ in~\eqref{eq:min:coupling:ab} are exchanged.
The update rule for~$M_{k+1}$ is:
\begin{equation}\label{eq:min:coupling:min}
M_{k+1}
=
\begin{cases}
M_k + 1 & \text{if } \xi_{k+1} \in \left[ 0 \,,\, \frac{\gamma M_k}{\Lambda(A_k,B_k,M_k)} \right)\\
M_k - 1 & \text{if } \xi_{k+1} \in \left[ 1 - \frac{\delta M_k^2}{\Lambda(A_k,B_k,M_k)} \,,\, 1 \right)\\
M_k & \text{otherwise}
\end{cases}
\end{equation}

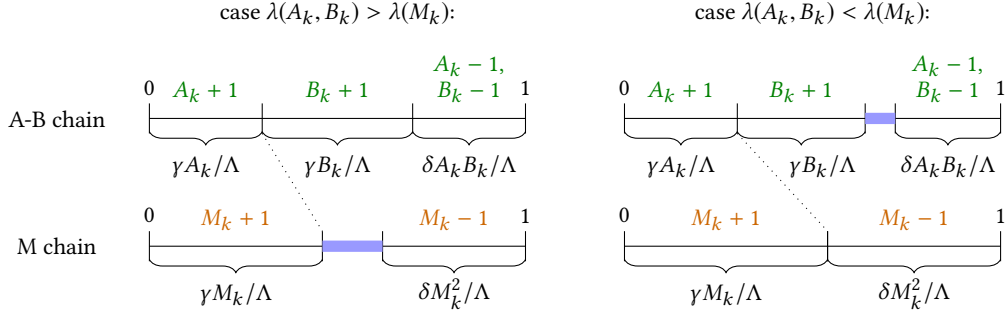
\begin{figure}
\begin{tikzpicture}
\coordinate (AB0) at (0,0);
\coordinate (Ak) at (1.5,0);
\coordinate (Bk) at (3.5,0);
\coordinate (AB1) at (5,0);

\coordinate (M0) at (0,-1.7);
\coordinate (Mk) at (2.3,-1.7);
\coordinate (Mk2) at (3.1,-1.7);
\coordinate (M1) at (5,-1.7);

% case
\node at ($ (AB0) + (2.5,1.4)$) {case $\lambda(A_k,B_k) > \lambda(M_k)$:};

% AB chain
\draw (AB0) -- (AB1) node [pos=0,xshift=-35pt] {A-B chain};
\draw ($ (AB0) + (0,-0.2) $) -- ++(0,0.4) node[above] {$0$};
\draw ($ (Ak) + (0,-0.2) $) -- ++(0,0.4);
\draw ($ (Bk) + (0,-0.2) $) -- ++(0,0.4);
\draw ($ (AB1) + (0,-0.2) $) -- ++(0,0.4) node[above] {$1$};

\draw [decorate,decoration={brace,amplitude=5pt,mirror,raise=5pt}] (AB0) -- (Ak) node [midway,yshift=-18pt] {$\gamma A_k/\Lambda$};

\draw [decorate,decoration={brace,amplitude=5pt,mirror,raise=5pt}] (Ak) -- (Bk) node [midway,yshift=-18pt] {$\gamma B_k/\Lambda$};

\draw [decorate,decoration={brace,amplitude=5pt,mirror,raise=5pt}] (Bk) -- (AB1) node [midway,yshift=-18pt] {$\delta A_k B_k/\Lambda$};

\path (AB0) -- (Ak) node[pos=0.5,yshift=10pt,green!50!black]
  {$A_k + 1$};
\path (Ak) -- (Bk) node[pos=0.5,yshift=10pt,green!50!black]
  {$B_k + 1$};
\path (Bk) -- (AB1) node[pos=0.5,yshift=20pt,green!50!black]
  {$A_k - 1$,};
\path (Bk) -- (AB1) node[pos=0.5,yshift=10pt,green!50!black]
  {$B_k - 1$};

% M chain
\draw (M0) -- (M1) node [pos=0,xshift=-35pt] {M chain};
\draw ($ (M0) + (0,-0.2) $) -- ++(0,0.4) node[above] {$0$};
\draw ($ (Mk) + (0,-0.2) $) -- ++(0,0.4);
\draw ($ (Mk2) + (0,-0.2) $) -- ++(0,0.4);
\draw ($ (M1) + (0,-0.2) $) -- ++(0,0.4) node[above] {$1$};

\draw [decorate,decoration={brace,amplitude=5pt,mirror,raise=5pt}] (M0) -- (Mk) node [midway,yshift=-18pt] {$\gamma M_k/\Lambda$};

\draw [decorate,decoration={brace,amplitude=5pt,mirror,raise=5pt}] (Mk2) -- (M1) node [midway,yshift=-18pt] {$\delta M_k^2/\Lambda$};

\path (M0) -- (Mk) node[pos=0.5,yshift=10pt,orange!80!black]
  {$M_k + 1$};
\path (Mk2) -- (M1) node[pos=0.5,yshift=10pt,orange!80!black]
  {$M_k - 1$};
  
% proven
\draw[dotted] ($(Ak) + (0,-0.2)$) -- ($(Mk) + (0,0.2)$);

% stuttering
\draw[-,line width=1.5mm,blue!40!white] (Mk) -- (Mk2);

\end{tikzpicture}
\hspace{0.8cm}
\begin{tikzpicture}
\coordinate (AB0) at (0,0);
\coordinate (Ak) at (1.5,0);
\coordinate (Bk) at (3.2,0);
\coordinate (AkBk) at (3.6,0);
\coordinate (AB1) at (5,0);

\coordinate (M0) at (0,-1.7);
\coordinate (Mk) at (2.7,-1.7);
\coordinate (Mk2) at (2.7,-1.7);
\coordinate (M1) at (5,-1.7);

% case
\node at ($ (AB0) + (2.5,1.4)$) {case $\lambda(A_k,B_k) < \lambda(M_k)$:};

% AB chain
\draw (AB0) -- (AB1);
\draw ($ (AB0) + (0,-0.2) $) -- ++(0,0.4) node[above] {$0$};
\draw ($ (Ak) + (0,-0.2) $) -- ++(0,0.4);
\draw ($ (Bk) + (0,-0.2) $) -- ++(0,0.4);
\draw ($ (AkBk) + (0,-0.2) $) -- ++(0,0.4);
\draw ($ (AB1) + (0,-0.2) $) -- ++(0,0.4) node[above] {$1$};

\draw [decorate,decoration={brace,amplitude=5pt,mirror,raise=5pt}] (AB0) -- (Ak) node [midway,yshift=-18pt] {$\gamma A_k/\Lambda$};

\draw [decorate,decoration={brace,amplitude=5pt,mirror,raise=5pt}] (Ak) -- (Bk) node [midway,yshift=-18pt,xshift=8pt] {$\gamma B_k/\Lambda$};

\draw [decorate,decoration={brace,amplitude=5pt,mirror,raise=5pt}] (AkBk) -- (AB1) node [midway,yshift=-18pt] {$\delta A_k B_k/\Lambda$};

\path (AB0) -- (Ak) node[pos=0.5,yshift=10pt,green!50!black]
  {$A_k + 1$};
\path (Ak) -- (Bk) node[pos=0.5,yshift=10pt,green!50!black]
  {$B_k + 1$};
\path (AkBk) -- (AB1) node[pos=0.5,yshift=20pt,green!50!black]
  {$A_k - 1$,};
\path (AkBk) -- (AB1) node[pos=0.5,yshift=10pt,green!50!black]
  {$B_k - 1$};

% M chain
\draw (M0) -- (M1);
\draw ($ (M0) + (0,-0.2) $) -- ++(0,0.4) node[above] {$0$};
\draw ($ (Mk) + (0,-0.2) $) -- ++(0,0.4);
\draw ($ (Mk2) + (0,-0.2) $) -- ++(0,0.4);
\draw ($ (M1) + (0,-0.2) $) -- ++(0,0.4) node[above] {$1$};

\draw [decorate,decoration={brace,amplitude=5pt,mirror,raise=5pt}] (M0) -- (Mk) node [midway,yshift=-18pt] {$\gamma M_k/\Lambda$};

\draw [decorate,decoration={brace,amplitude=5pt,mirror,raise=5pt}] (Mk2) -- (M1) node [midway,yshift=-18pt] {$\delta M_k^2/\Lambda$};

\path (M0) -- (Mk) node[pos=0.5,yshift=10pt,orange!80!black]
  {$M_k + 1$};
\path (Mk2) -- (M1) node[pos=0.5,yshift=10pt,orange!80!black]
  {$M_k - 1$};
  
% proven
\draw[dotted] ($(Ak) + (0,-0.2)$) -- ($(Mk) + (0,0.2)$);  
  
% stuttering
\draw[-,line width=1.5mm,blue!40!white] (Bk) -- (AkBk);

\end{tikzpicture}
\caption{Continuous-time coupling of the A-B chain and the single-species birth-death M-chain, given that $A_k \leq B_k$, with $\Lambda = \Lambda(A_k,B_k,M_k)$. The intervals for the cases of $\xi_{k+1}$ and their effect on the A-B chain and the M-chain are shown in green and orange, respectively. Cases that lead to stuttering steps are shown in blue. The dotted relation between intervals is proven in Lemma~\ref{lem:coupling:ineq}.}
\label{fig:coupling_AB_M}
\end{figure}

\paragraph{Analysis for time until consensus.}
Note, that in the coupling ``stuttering steps'' for~$(A_k,B_k)$ or~$M_k$ are possible in the definition of the coupled process,
  making the underlying discrete-time jump chains of, e.g., chain~$(A(t),B(t))$ and the A-B protocol, potentially differ.
Indeed, the event $(A_{k+1},B_{k+1}) = (A_k,B_k)$ is possible with positive probability if $\lambda(A_k,B_k) < \lambda(M_k)$, and $M_{k+1}=M_k$ has positive probability if $\lambda(M_k) < \lambda(A_k,B_k)$; see \figref{fig:coupling_AB_M}.
The following elementary Lemma~\ref{lem:markov:stuttering}, however, shows that the continuous-time chain~$(A(t),B(t))$ and the A-B protocol chain have identical transition rates, and are thus identically distributed.
The same holds true for the continuous-time chain~$M(t)$ and the birth-death M chain.
\opt{append}{Its proof is postponed to the appendix.}

\begin{lemrep}\label{lem:markov:stuttering}
Let~$T_1,T_2,\dots$ be a sequence of  i.i.d.\ exponential random variables with rate parameter~$\lambda$ and let~$k$ be an independent geometric random variable with success probability~$p$.
Then $T=T_1+\dots+T_k$
is exponentially distributed with rate parameter~$p\lambda$.
\end{lemrep}
\begin{proof}
By the law of total probability, for every $t\geq0$, we have
\begin{equation*}
\begin{split}
\IP(T \leq t)
& =
\sum_{k=0}^\infty p (1-p)^{k} \IP(T_1+\cdots+ T_{k+1} \leq t)
 =
\sum_{k=0}^\infty p (1-p)^{k} e^{-\lambda t} \sum_{i=k+1}^\infty \frac{1}{i!} (\lambda t)^i
\\ & =
e^{-\lambda t}
\sum_{i=0}^\infty \frac{1}{i!} (\lambda t)^i 
p \sum_{k=0}^{i-1} (1-p)^k
=
e^{-\lambda t}
\sum_{i=0}^\infty \frac{1}{i!} (\lambda t)^i (1 - (1-p)^i)
\\ & =
e^{-\lambda t} ( e^{\lambda t} - e^{(1-p)\lambda t})
=
1 - e^{-p\lambda t}
\enspace,
\end{split}
\end{equation*}
which is equal to the cumulative distribution function of an exponential random variable with parameter~$p\lambda$.
\end{proof}

By construction of the coupled process, the single-species birth-death process~$M(t)$ indeed dominates the minimum of the species population counts~$A(t)$ and~$B(t)$ in the following way:

\begin{lem}\label{lem:coupling:ineq}
In the coupled process, $\min\{A(t),B(t)\}\leq M(t)$ for all times $t\geq0$.
\end{lem}
\begin{proof}
Let~$K$ be the step number of the discrete-time coupled process such that $t_K \leq t < t_{K+1}$, where~$t_k$ is the time of the $k$\textsuperscript{th} step.
We show by induction that $\min\{A_k,B_k\} \leq M_k$  for all $k\in\IN$.
The inequality holds initially, for $k=0$, by definition of the coupled process.
Now assume that $\min\{A_k,B_k\}\leq M_k$.
Without loss of generality, by symmetry, assume that $A_k\leq B_k$, so that $A_k = \min\{A_k,B_k\} \leq M_k$.
Then $\gamma A_k \leq \gamma M_k$ and thus $A_{k+1} = A_k + 1$ implies $M_{k+1} = M_k + 1$ by the definition of the coupling in~\eqref{eq:min:coupling:ab} and~\eqref{eq:min:coupling:min}; see \figref{fig:coupling_AB_M}.
We distinguish the two cases $A_k < M_k$ and $A_k = M_k$.

If $A_k < M_k$, then the only way to have $A_{k+1} > M_{k+1}$ is to have $A_{k+1} = A_k + 1$ and $M_{k+1} = M_k - 1$. But this is impossible since $A_{k+1} = A_k + 1$ implies $M_{k+1} = M_k + 1$.

Otherwise, $A_k = M_k$. The case is shown in \figref{fig:coupling_AB_M:case2}.
We have, $\delta M_k^2 = \delta A_k^2 \leq \delta A_k B_k$.
Thus $M_{k+1} = M_k - 1$ implies $A_{k+1} = A_k - 1$ and $B_{k+1} = B_k - 1$.
Hence, combined with the above implication which remains true, we have $A_{k+1} \leq M_{k+1}$ in all possible cases for~$\xi_{k+1}$.
\end{proof}

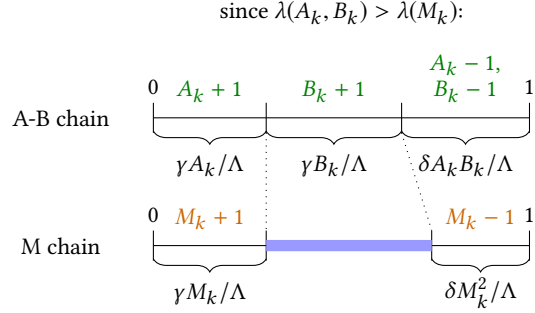
\begin{figure}[h]
\begin{tikzpicture}
\coordinate (AB0) at (0,0);
\coordinate (Ak) at (1.5,0);
\coordinate (Bk) at (3.3,0);
\coordinate (AB1) at (5,0);

\coordinate (M0) at (0,-1.7);
\coordinate (Mk) at (1.5,-1.7);
\coordinate (Mk2) at (3.7,-1.7);
\coordinate (M1) at (5,-1.7);

% case
\node at ($ (AB0) + (2.5,1.4)$) {since $\lambda(A_k,B_k) > \lambda(M_k)$:};

% AB chain
\draw (AB0) -- (AB1) node [pos=0,xshift=-35pt] {A-B chain};
\draw ($ (AB0) + (0,-0.2) $) -- ++(0,0.4) node[above] {$0$};
\draw ($ (Ak) + (0,-0.2) $) -- ++(0,0.4);
\draw ($ (Bk) + (0,-0.2) $) -- ++(0,0.4);
\draw ($ (AB1) + (0,-0.2) $) -- ++(0,0.4) node[above] {$1$};

\draw [decorate,decoration={brace,amplitude=5pt,mirror,raise=5pt}] (AB0) -- (Ak) node [midway,yshift=-18pt] {$\gamma A_k/\Lambda$};

\draw [decorate,decoration={brace,amplitude=5pt,mirror,raise=5pt}] (Ak) -- (Bk) node [midway,yshift=-18pt] {$\gamma B_k/\Lambda$};

\draw [decorate,decoration={brace,amplitude=5pt,mirror,raise=5pt}] (Bk) -- (AB1) node [midway,yshift=-18pt] {$\delta A_k B_k/\Lambda$};

\path (AB0) -- (Ak) node[pos=0.5,yshift=10pt,green!50!black]
  {$A_k + 1$};
\path (Ak) -- (Bk) node[pos=0.5,yshift=10pt,green!50!black]
  {$B_k + 1$};
\path (Bk) -- (AB1) node[pos=0.5,yshift=20pt,green!50!black]
  {$A_k - 1$,};
\path (Bk) -- (AB1) node[pos=0.5,yshift=10pt,green!50!black]
  {$B_k - 1$};

% M chain
\draw (M0) -- (M1) node [pos=0,xshift=-35pt] {M chain};
\draw ($ (M0) + (0,-0.2) $) -- ++(0,0.4) node[above] {$0$};
\draw ($ (Mk) + (0,-0.2) $) -- ++(0,0.4);
\draw ($ (Mk2) + (0,-0.2) $) -- ++(0,0.4);
\draw ($ (M1) + (0,-0.2) $) -- ++(0,0.4) node[above] {$1$};

\draw [decorate,decoration={brace,amplitude=5pt,mirror,raise=5pt}] (M0) -- (Mk) node [midway,yshift=-18pt] {$\gamma M_k/\Lambda$};

\draw [decorate,decoration={brace,amplitude=5pt,mirror,raise=5pt}] (Mk2) -- (M1) node [midway,yshift=-18pt] {$\delta M_k^2/\Lambda$};

\path (M0) -- (Mk) node[pos=0.5,yshift=10pt,orange!80!black]
  {$M_k + 1$};
\path (Mk2) -- (M1) node[pos=0.5,yshift=10pt,orange!80!black]
  {$M_k - 1$};
  
% proven
\draw[dotted] ($(Ak) + (0,-0.2)$) -- ($(Mk) + (0,0.2)$);
\draw[dotted] ($(Bk) + (0,-0.2)$) -- ($(Mk2) + (0,0.2)$);

% stuttering
\draw[-,line width=1.5mm,blue!40!white] (Mk) -- (Mk2);

\end{tikzpicture}
\caption{The case $A_k=M_k$ in the proof of Lemma~\ref{lem:coupling:ineq}, with $\Lambda = \Lambda(A_k,B_k,M_k)$.
The case's assumption implies that $\lambda(A_k,B_k) > \lambda(M_k)$.
The dotted relation between intervals is shown in the proof.}
\label{fig:coupling_AB_M:case2}
\end{figure}

Lemma~\ref{lem:coupling:ineq} allows to compare the probabilities of extinction in the single-species chain and of consensus in the A-B protocol chain:
\begin{lem}\label{lem:coupling:probs}
$\IP(M(t) = 0) \leq \IP(A(t) = 0 \vee B(t) = 0)$ for all times $t\geq 0$.
\end{lem}

It thus suffices to prove bounds on the time until the population goes extinct in the single-species M chain.
For that, we leverage known results on birth-death processes, which are not applicable to the two-species A-B protocol chain.

\begin{lem}\label{lem:min:expected:time}
If\/ $T$ denotes the time until extinction in the single-species process~$M(t)$, then
$\IE\,T = O(1)$.
\end{lem}
\begin{proof}
The birth rate in state $M(t) = i$
is equal to $\alpha(i) = i\gamma$ and the death rate is equal to $\beta(i) = i^2 \delta$.
From known general results on birth-death process \cite[p.~149]{karlin75} we obtain,
when starting from initial population $M(0) = M$, that
\begin{equation}
\begin{split}
\IE\,T
& =
\sum_{j=1}^M \sum_{k=j-1}^\infty \frac{\alpha(j)\cdots \alpha(k)}{\beta(j)\cdots \beta(k)}\cdot \frac{1}{\beta(k+1)}
 =
\sum_{j=1}^M \sum_{k=j-1}^\infty \frac{\gamma^{k-j+1}}{\delta^{k-j+1}k!/(j-1)!}\cdot\frac{1}{(k+1)^2\delta}
\end{split}
\end{equation}
Setting $\alpha = \gamma/\delta$, we have
\begin{equation}\label{eq:expected:time:bound}
\begin{split}
\IE\,T
& =
\frac{1}{\delta}
\sum_{j=1}^M \sum_{k=j-1}^\infty
\alpha^{k-j+1} 
\frac{(j-1)!}{(k+1)!(k+1)}
 =
\frac{1}{\delta}
\sum_{j=1}^M
\frac{(j-1)!}{\alpha^{j}}
\sum_{k=j}^\infty
\frac{\alpha^{k}}{k!k}
\\ & =
\frac{1}{\delta}
\sum_{j=1}^M
\frac{(j-1)!}{\alpha^{j}} \cdot
\frac{\alpha^{j}}{j!j}\sum_{k=j}^\infty
\frac{\alpha^{k-j}}{k!/j! \cdot k/j}
\leq
\frac{1}{\delta}
\sum_{j=1}^M
\frac{(j-1)!}{\alpha^{j}} \cdot
\frac{\alpha^{j}}{j!j}\sum_{k=j}^\infty
\frac{\alpha^{k-j}}{(k-j)!}
\end{split}
\end{equation}
since for $k \ge j \ge 1$, it is $k!/j! \ge (k-j)!$ and $k/j \ge 1$.
Thus,
\begin{equation}
\begin{split}
\IE\,T & \leq
\frac{1}{\delta}
\sum_{j=1}^M
\frac{(j-1)!}{\alpha^{j}} \cdot
\frac{\alpha^{j}}{j!j} \cdot
e^{\alpha}
 = 
\frac{e^{\alpha}}{\delta}
\sum_{j=1}^M \frac{1}{j^2}
 \leq \frac{e^{\alpha}\pi^2}{6\delta} = 
O(1)
\enspace.
\end{split}
\end{equation}
This concludes the proof.
\end{proof}

Denoting with $T_{AB}$ the earliest time $t$ such that $A(t)=0$ or $B(t) = 0$, and with $T_M$ the earliest time $t$ such that $M(t) = 0$, Lemma~\ref{lem:coupling:probs} is equivalent to
$\IP(T_M \leq t ) \leq \IP(T_{AB} \leq t)$, which, in turn, is equivalent to
$\IP(T_M > t )  \geq \IP(T_{AB} > t)$.
Using the formula $\IE\,T = \int_0^\infty \IP(T > t)\, dt$, we further have
\begin{equation}
\IE(T_M) = \int_0^\infty \IP(T_M > t)\, dt \geq \int_0^\infty \IP(T_{AB} > t)\, dt = \IE(T_{AB})\enspace.
\end{equation}
Combining this with Lemma~\ref{lem:min:expected:time}, shows that the expected time until consensus in the A-B protocol is also~$O(1)$.
For the high-probability result in the first part of Theorem~\ref{thm:consensus}, we simply make $\Theta(\log n)$ consecutive tries to achieve extinction in an interval of constant time:

\begin{lem}\label{lem:min:whp:time}
If~$T$ denotes the time until extinction in the singles-species process~$M(t)$, then there exists a constant~$C$ such that $\IP(T \leq C\log_2 n) = 1 - O(1/n)$.
\end{lem}
\begin{proof}
Let~$C_1$ be the~$O(1)$ constant from Lemma~\ref{lem:min:expected:time} and set $C=\max\{2C_1,2\}$.
Then, by Markov's inequality, we have 
$\IP(T > C) \leq C_1/C = 1/2$.
Thus, the probability of the event $T > C\log_2 n$ is dominated by the probability of failing $\log_2 n$ consecutive tries with a Bernoulli random variable with parameter $p = 1/2$.
But this probability is $2^{-\log_2 n} = 1/n$.
\end{proof}

A simple combination of Lemmas~\ref{lem:coupling:probs} and~\ref{lem:min:whp:time} completes the proof of the first part of Theorem~\ref{thm:consensus}.

\subsection{Probability of Reaching Majority Consensus}

We now turn to the proof of the second part of Theorem~\ref{thm:consensus}, i.e., the bound on the probability to achieve majority consensus.
We use a coupling of the A-B protocol chain with a different process than for the time bound.
Namely we couple it with two parallel independent Yule processes.
A Yule process, also known as a pure birth process, has this single state-transition rule $X\to X+1$ with linear transition rate~$\gamma X$.
Since we already showed the upper bound on the time until consensus, it suffices to look at the discrete-time jump process.
In particular, the coupling we define is discrete-time.

\paragraph{Discrete-time coupling.}
For an i.i.d.\ sequence $(\xi_k)_{k\geq 1}$ of uniformly distributed random variables in the unit interval $[0,1)$, we
define the coupled process $(A_k,B_k,X_k,Y_k)$ by
$A_0=X_0$, $B_0=Y_0$, and
\begin{equation}\label{eq:value:coupling}
\begin{split}
(A_{k+1},B_{k+1})
& =
\begin{cases}
(A_k-1,B_k-1) & \text{if } \xi_{k+1} \in \left[0 \,,\, \frac{\delta A_k B_k}{\gamma(A_k+B_k)+\delta A_k B_k} \right)\\
(A_k+1,B_k) & \text{if } \xi_{k+1} \in \left[\frac{\delta A_k B_k}{\gamma(A_k+B_k)+\delta A_k B_k}\,,\, 1 - \frac{\gamma B_k}{\gamma(A_k+B_k)+\delta A_k B_k} \right)\\
(A_k,B_k+1) & \text{if } \xi_{k+1} \in \left[1 - \frac{\gamma B_k}{\gamma(A_k+B_k)+\delta A_k B_k}\,,\,1 \right)
\end{cases}
\\
(X_{k+1},Y_{k+1})
& =
\begin{cases}
(X_k,Y_k) & \text{if } \xi_{k+1} \in \left[0 \,,\, \frac{\delta A_k B_k}{\gamma(A_k+B_k)+\delta A_k B_k} \right)\\
(X_k+1,Y_k) & \text{if } \xi_{k+1} \in \left[\frac{\delta A_k B_k}{\gamma(A_k+B_k)+\delta A_k B_k}\,,\, 1 - \frac{\gamma (A_k+B_k)}{\gamma(A_k+B_k)+\delta A_k B_k} \cdot \frac{Y_k}{X_k+Y_k} \right)\\
(X_k,Y_k+1) & \text{if } \xi_{k+1} \in \left[1 - \frac{\gamma (A_k+B_k)}{\gamma(A_k+B_k)+\delta A_k B_k} \cdot \frac{Y_k}{X_k+Y_k} \,,\,1 \right)
\end{cases}
\end{split}
\end{equation}
if $\max\{A_k,B_k\}>0$ and $\max\{X_k,Y_k\}>0$.
Otherwise the process remains constant.
\figref{fig:coupling_AB_XY} visualizes the construction.

\begin{figure}
\begin{tikzpicture}
\coordinate (AB0) at (0,0);
\coordinate (Ak) at (2,0);
\coordinate (Bk) at (7.5,0);
\coordinate (AB1) at (10,0);

\coordinate (XY0) at (0,-1.6);
\coordinate (Xk) at (2,-1.6);
\coordinate (Yk) at (6.5,-1.6);
\coordinate (XY1) at (10,-1.6);

\coordinate (C0) at (0,-3.2);
\coordinate (C1) at (2,-3.2);
\coordinate (C2) at (6.5,-3.2);
\coordinate (C3) at (7.5,-3.2);
\coordinate (C4) at (10,-3.2);

% AB chain
\draw (AB0) -- (AB1) node [pos=0,xshift=-45pt] {A-B chain};
\draw ($ (AB0) + (0,-0.2) $) -- ++(0,0.4) node[above] {$0$};
\draw ($ (Ak) + (0,-0.2) $) -- ++(0,0.4);
\draw ($ (Bk) + (0,-0.2) $) -- ++(0,0.4);
\draw ($ (AB1) + (0,-0.2) $) -- ++(0,0.4) node[above] {$1$};

\draw [decorate,decoration={brace,amplitude=5pt,mirror,raise=5pt}] (AB0) -- (Ak) node [midway,yshift=-18pt,xshift=-10pt] {$\delta A_k B_k/\lambda(A_k,B_k)$};

\draw [decorate,decoration={brace,amplitude=5pt,mirror,raise=5pt}] (Ak) -- (Bk) node [midway,yshift=-18pt] {$\gamma A_k/\lambda(A_k,B_k)$};

\draw [decorate,decoration={brace,amplitude=5pt,mirror,raise=5pt}] (Bk) -- (AB1) node [midway,yshift=-18pt] {$\gamma B_k/\lambda(A_k,B_k)$};

\path (AB0) -- (Ak) node[pos=0.5,yshift=20pt,green!50!black]
  {$A_k - 1$,};
\path (AB0) -- (Ak) node[pos=0.5,yshift=10pt,green!50!black]
  {$B_k - 1$};
\path (Ak) -- (Bk) node[pos=0.5,yshift=10pt,green!50!black]
  {$A_k + 1$};
\path (Bk) -- (AB1) node[pos=0.5,yshift=10pt,green!50!black]
  {$B_k + 1$,};

% XY chain
\draw (XY0) -- (XY1) node [pos=0,xshift=-45pt] {Yule processes X, Y};
\draw ($ (XY0) + (0,-0.2) $) -- ++(0,0.4) node[above] {$0$};
\draw ($ (Xk) + (0,-0.2) $) -- ++(0,0.4);
\draw ($ (Yk) + (0,-0.2) $) -- ++(0,0.4);
\draw ($ (XY1) + (0,-0.2) $) -- ++(0,0.4) node[above] {$1$};

\draw [decorate,decoration={brace,amplitude=5pt,mirror,raise=5pt}] (XY0) -- (Xk) node [midway,yshift=-18pt,xshift=-10pt] {$\delta A_k B_k/\lambda(A_k,B_k)$};

\draw [decorate,decoration={brace,amplitude=5pt,mirror,raise=5pt}] (Xk) -- (Yk) node [midway,yshift=-18pt] {$(1-\delta A_k B_k/\lambda(A_k,B_k)) \cdot \frac{X_k}{X_k+Y_k}$};

\draw [decorate,decoration={brace,amplitude=5pt,mirror,raise=5pt}] (Yk) -- (XY1) node [midway,yshift=-18pt,xshift=15pt] {$(1-\delta A_k B_k/\lambda(A_k,B_k)) \cdot \frac{Y_k}{X_k+Y_k}$};

\path (Xk) -- (Yk) node[pos=0.5,yshift=10pt,orange!80!black]
  {$X_k + 1$};
\path (Yk) -- (XY1) node[pos=0.5,yshift=10pt,orange!80!black]
  {$Y_k + 1$};
  
% proven
\draw[dotted] ($(Ak) + (0,-0.2)$) -- ($(Xk) + (0,0.2)$);
\draw[dotted] ($(Bk) + (0,-0.2)$) -- ($(Yk) + (0,0.2)$);

% stuttering
\draw[-,line width=1.5mm,blue!40!white] (XY0) -- (Xk);

% proof cases
\draw (C0) -- (C4) node [pos=0,xshift=-45pt,yshift=5pt] {Cases in proof} node[pos=0,xshift=-45pt,yshift=-5pt] {of Lemma~\ref{lem:ab:yule:coupling}};
\draw ($ (C0) + (0,-0.2) $) -- ++(0,0.4);
\draw ($ (C1) + (0,-0.2) $) -- ++(0,0.4);
\draw ($ (C2) + (0,-0.2) $) -- ++(0,0.4);
\draw ($ (C3) + (0,-0.2) $) -- ++(0,0.4);
\draw ($ (C4) + (0,-0.2) $) -- ++(0,0.4);

\path (C0) -- (C1)
  node[midway,yshift=-10pt] {(1)};
\path (C1) -- (C2)
  node[midway,yshift=-10pt] {(2)};
\path (C2) -- (C3)
  node[midway,yshift=-10pt] {(3)};
\path (C3) -- (C4)
  node[midway,yshift=-10pt] {(4)};
  
\end{tikzpicture}
\caption{Discrete-time coupling of A-B chain and two Yule processes X and Y with $\lambda(A_k,B_k) = \gamma(A_k + B_k) + \delta A_k B_k$. Cases for $\xi_{k+1}$ that lead to stuttering steps are shown in blue. The interval relations indicated by the dotted lines are proven by induction in Lemma~\ref{lem:ab:yule:coupling}.
The four cases for the induction step are indicated.}
\label{fig:coupling_AB_XY}
\end{figure}
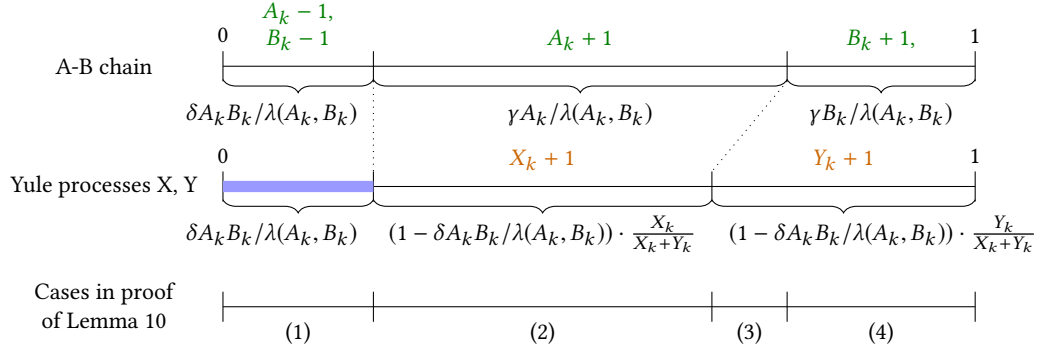

\paragraph{Analysis for probability of reaching majority consensus.}
We start with two simple technical lemmas that we will use for the comparison of the coupled processes.

\begin{lem}\label{lem:ratios:conversion}
Let $a,b,x,y\in\IR_0^+$ with $\max\{a,b\}>0$ and $\max\{x,y\}>0$.
Then $\frac{b}{a+b} \leq \frac{y}{x+y}$ if and only if $bx\leq ay$.
\end{lem}
\begin{proof}
Multiplying both sides by $(a+b)\cdot (x+y)$, we see that the first inequality is equivalent to $bx+by\leq ay+by$, which is in turn equivalent to $bx \leq ay$.
\end{proof}

\begin{lem}\label{lem:diff:to:ratio}
Let $a,b,x,y,m\in\IR_0^+$ with $\max\{a,b\}>0$, $\max\{x,y\}>0$, and $x\geq y$.
If $x \leq a + m$ and $y \geq b + m$, then~$\frac{b}{a+b} \leq \frac{y}{x+y}$.
\end{lem}
\begin{proof}
By Lemma~\ref{lem:ratios:conversion} it suffices to prove $bx\leq ay$.
From the inequality chain $a+m \geq x \geq y \geq b+m$ we get $a\geq b$.
We thus have
$bx \leq b(a+m) = ab + bm \leq ab + am = a(b+m) \leq ay$.
\end{proof}

The crucial property of this coupling is that the initial minority in the A-B process cannot overtake the initial majority before the initial minority overtakes the initial majority in the parallel Yule processes.
We now prove that our construction indeed has this property.

\begin{lem}\label{lem:ab:yule:coupling}
If $X_0 = A_0 \geq B_0 = Y_0$ and $X_k \geq Y_k$ for all $0\leq k\leq K$, then
$X_k - Y_k \leq A_k - B_k$ for all $0\leq k\leq K$.
\end{lem}
\begin{proof}
We first show by induction on~$k$ that $X_k \leq A_k+m_k$ and $Y_k \geq B_k+m_k$ for all $0\leq k\leq K$, where~$m_k$ is the number of death reactions up to step~$k$.
In the base case $k=0$ we even have equality.
For the induction step $k\mapsto k+1$,
we distinguish four cases; see \figref{fig:coupling_AB_XY}.
\begin{enumerate}
\item $\xi_{k+1} \in \left[0 \,,\, \frac{\delta A_k B_k}{\gamma(A_k+B_k)+\delta A_k B_k} \right)$:
Then $m_{k+1} = m_k + 1$, $A_{k+1}=A_k-1$, $B_{k+1}=B_k-1$, $X_{k+1}=X_k$, and $Y_{k+1}=Y_k$.
Hence,
$X_{k+1} = X_k \leq A_k + m_k = (A_{k+1} + 1) + m_k = A_{k+1} + m_{k+1}$
and
$Y_{k+1} = Y_k \geq B_k + m_k = (B_{k+1} + 1) + m_k = B_{k+1} + m_{k+1}$
by the induction hypothesis.

\item $\xi_{k+1} \in \left[\frac{\delta A_k B_k}{\gamma(A_k+B_k)+\delta A_k B_k}\,,\, 1 - \frac{\gamma (A_k+B_k)}{\gamma(A_k+B_k)+\delta A_k B_k} \cdot \frac{Y_k}{X_k+Y_k} \right)$:
In particular we have
\begin{equation}
\begin{split}
\xi_{k+1}
& \leq
1 - \frac{\gamma (A_k+B_k)}{\gamma(A_k+B_k)+\delta A_k B_k} \cdot \frac{Y_k}{X_k+Y_k}
\\ & \leq
1 - \frac{\gamma (A_k+B_k)}{\gamma(A_k+B_k)+\delta A_k B_k} \cdot \frac{B_k}{A_k+B_k}
\\ & =
1 - \frac{\gamma B_k}{\gamma(A_k+B_k)+\delta A_k B_k}
\end{split}
\end{equation}
by the induction hypothesis and Lemma~\ref{lem:diff:to:ratio}.
This implies the interval relation indicated in \figref{fig:coupling_AB_XY}.

Hence, $m_{k+1} = m_k$, $A_{k+1} = A_k+1$, $B_{k+1} = B_k$, $X_{k+1} = X_k + 1$, and $Y_{k+1} = Y_k$.
But this means
$X_{k+1} = X_k + 1 \leq A_k + m_k + 1 = (A_k + 1) + m_k = A_{k+1} + m_{k+1}$
and
$Y_{k+1} = Y_k \geq B_k + m_k = B_{k+1} + m_{k+1}$
by the induction hypothesis.

\item $\xi_{k+1} \in \left[1 - \frac{\gamma (A_k+B_k)}{\gamma(A_k+B_k)+\delta A_k B_k} \cdot \frac{Y_k}{X_k+Y_k} \,,\,  1 - \frac{\gamma B_k}{\gamma(A_k+B_k)+\delta A_k B_k} \right)$:
We have $m_{k+1} = m_k$, $A_{k+1} = A_k+1$, $B_{k+1} = B_k$, $X_{k+1} = X_k $, and $Y_{k+1} = Y_k+1$.
But this means
$X_{k+1} = X_k \leq X_k + 1 \leq A_k + m_k + 1 = (A_k + 1) + m_k = A_{k+1} + m_{k+1}$
and
$Y_{k+1} = Y_k +1 \geq Y_k \geq B_k + m_k = B_{k+1} + m_{k+1}$
by the induction hypothesis.

\item $\xi_{k+1} \in \left[1 - \frac{\gamma B_k}{\gamma(A_k+B_k)+\delta A_k B_k}\,,\, 1 \right)$:
In particular we have
\begin{equation}
\begin{split}
\xi_{k+1}
& \geq
1 - \frac{\gamma B_k}{\gamma(A_k+B_k)+\delta A_k B_k}
\\ & =
1 - \frac{\gamma (A_k+B_k)}{\gamma(A_k+B_k)+\delta A_k B_k} \cdot \frac{B_k}{A_k+B_k}
\\ & \geq
1 - \frac{\gamma (A_k+B_k)}{\gamma(A_k+B_k)+\delta A_k B_k} \cdot \frac{Y_k}{X_k+Y_k}
\end{split}
\end{equation}
by the induction hypothesis and Lemma~\ref{lem:diff:to:ratio}.
Hence $m_{k+1} = m_k$, $A_{k+1} = A_k$, $B_{k+1} = B_k + 1$, $X_{k+1} = X_k$, and $Y_{k+1} = Y_k + 1$.
But this means
$X_{k+1} = X_k \leq A_k + m_k = A_{k+1} + m_{k+1}$
and
$Y_{k+1} = Y_k + 1 \geq B_k + m_k + 1 = (B_k + 1) + m_k = B_{k+1} + m_{k+1}$
by the induction hypothesis.

The lemma now follows via $X_k - Y_k \leq (A_k + m_k) - (B_k + m_k) = A_k - B_k$.
\end{enumerate}
\end{proof}

\begin{lem}\label{lem:collision:comparison}
If $A_0=X_0$ and $B_0=Y_0$, then
$\IP(\exists k\colon A_k = B_k) \leq \IP(\exists k\colon X_k = Y_k)$.
\end{lem}
\begin{proof}
By Lemma~\ref{lem:ab:yule:coupling}, if~$k$ is minimal such that $A_k = B_k$, then $X_k = Y_k$.
\end{proof}

As defined in the coupling the parallel Yule processes $(X_k,Y_k)$ can have stuttering steps where $(X_{k+1},Y_{k+1})=(X_k,Y_k)$.
However, this happens only finitely often almost surely.
This allows us to analyze a version of  the process~$(X_k,Y_k)$ without stuttering steps in the sequel.

\begin{lem}\label{lem:slow:yule}
If $(\tilde{X}_k, \tilde{Y}_k)$ is the product of two independent pure-birth processes with $\tilde{X}_0 = X_0$ and $\tilde{Y}_0 = Y_0$, then $\IP(\exists k\colon \tilde{X}_k = \tilde{Y}_k) = \IP(\exists k \colon X_k = Y_k)$.
\end{lem}
\begin{proof}
Lemma~\ref{lem:min:expected:time} implies that there are only finitely many deaths in the coupled chain almost surely.
There are hence only finitely many stuttering steps in $(X_k,Y_k)$ almost surely.
\end{proof}

By slight abuse of notation, we will use~$(X_k,Y_k)$ to refer to the parallel Yule processes without any stuttering steps.

Two parallel independent Yule processes are known to be related to a beta distribution, which we will use below.
The regularized incomplete beta function $I_z(\alpha,\beta)$ is defined as
\begin{equation}
I_z(\alpha,\beta)
=
\int_0^z t^{\alpha-1} (1-t)^{\beta-1}\,dt
\ \Big/
\int_0^1 t^{\alpha-1} (1-t)^{\beta-1}\,dt
\enspace.
\end{equation}

\begin{lem}\label{lem:collision:yule}
If $X_0 > Y_0$,
then
$\IP\left(\exists k\colon X_k = Y_k\right) = 2\cdot I_{1/2}(X_0,Y_0)$.
\end{lem}
\begin{proof}
The sequence of ratios $\frac{X_k}{X_k+Y_k}$ converges with probability~$1$ and the limit is distributed according to a beta distribution with parameters $\alpha = X_0$ and $\beta = Y_0$ \cite[Theorem~3.2]{mahmoud09polya}.
In particular, the probability that the limit is less than~$1/2$ is equal to the beta distribution's cumulative distribution function evaluated at~$1/2$, i.e., equal to $I_{1/2}(X_0,Y_0)$.
Because initially we have $X_0>Y_0$, the law of total probability gives:
\begin{equation}\label{eq:collision:yule:total}
\begin{split}
I_{1/2}(X_0,Y_0)
=
\IP\left( \lim_{k\to\infty} \frac{X_k}{X_k+Y_k} < \frac{1}{2} \right)
=\ &
\IP\left( \lim_{k\to\infty} \frac{X_k}{X_k+Y_k} < \frac{1}{2} \ \Big|\ \exists k\colon X_k=Y_k \right)
\cdot
\IP\left( \exists k\colon X_k=Y_k \right)
\\
& +
\IP\left( \lim_{k\to\infty} \frac{X_k}{X_k+Y_k} < \frac{1}{2} \ \wedge\  \forall k\colon X_k > Y_k \right)
\end{split}
\end{equation}
Now, if $\forall k\colon X_k > Y_k$, then $\lim_k \frac{X_k}{X_k+Y_k} \geq {1}/{2}$, which shows that the second term in the sum in~\eqref{eq:collision:yule:total} is zero.
Further, under the condition $\exists k\colon X_k=Y_k$, it is equiprobable for the limit of $\frac{X_k}{X_k+Y_k}$ to be larger or smaller than~$1/2$ by symmetry and the strong Markov property.
This shows that the right-hand side of~\eqref{eq:collision:yule:total} is equal to $\frac{1}{2}\cdot\IP\left(\exists k\colon X_k = Y_k\right)$.
But then $\IP\left(\exists k\colon X_k = Y_k\right) = 2\cdot I_{1/2}(X_0,Y_0)$, which concludes the proof.
\end{proof}

We define the event ``$B$ wins'' as~$A$ eventually becoming extinct.
Then, we have:

\begin{lem}\label{lem:collision:ab}
If $A_0 > B_0$, then $\IP\left(\exists k\colon A_k = B_k\right) = 2\cdot \IP(\text{$B$ wins})$.
\end{lem}
\begin{proof}
Similarly to the proof of Lemma~\ref{lem:collision:yule},
by the law of total probability, we have:
\begin{equation}\label{eq:collision:ab:total}
% \begin{split}
\IP\left( \text{$B$ wins} \right)
=
\IP\left( \text{$B$ wins} \mid \exists k\colon A_k=B_k \right)
\cdot
\IP\left( \exists k\colon A_k=B_k \right)
+
\IP\left( \text{$B$ wins} \wedge \forall k\colon A_k > B_k \right)
% \end{split}
\end{equation}
If $\forall k\colon A_k > B_k$, then $B$ cannot win, i.e., the second term in the right-hand side of~\eqref{eq:collision:ab:total} is zero.
Also, by symmetry and the strong Markov property, it is $\IP\left( \text{$B$ wins} \mid \exists k\colon A_k=B_k \right)=1/2$. 
A simple algebraic manipulation now concludes the proof.
\end{proof}

Combining the previous two lemmas with the coupling, we get an upper bound on the probability that the A-B protocol fails to reach majority consensus.
This upper bound is in terms of the regularized incomplete beta function.

\begin{lem}\label{lem:ab:beta:bound}
If $A_0 \geq B_0$, then
the A-B protocol fails to reach majority consensus with probability at most $I_{1/2}(A_0,B_0)$.
\end{lem}
\begin{proof}
Setting $X_0=A_0$ and $Y_0=B_0$, and
combining Lemmas~\ref{lem:collision:comparison}, \ref{lem:collision:yule}, and~\ref{lem:collision:ab}, we get
$\IP(\text{$B$ wins}) = \frac{1}{2}\cdot \IP(\exists k\colon A_k = B_k) \leq \frac{1}{2}\cdot \IP(\exists k\colon  X_k = Y_k) = I_{1/2}(A_0,B_0)$.
\end{proof}

Due to Lemma~\ref{lem:ab:beta:bound}, it only remains to upper-bound the term $I_{1/2}(\alpha,\beta)$.
Lemma~\ref{lem:beta:bound:delta} provides such a bound.
\opt{append}{Its proof is given in the appendix.}

\begin{lemrep}\label{lem:beta:bound:delta}
For $m,\Delta \in \IN$, it holds that
$\displaystyle I_{1/2}(m+\Delta,m) = \exp\left(-\Omega\left (\frac{\Delta^2}{m} \right) \right)$.
\end{lemrep}
\begin{proof}
We have the well-known formula
\begin{equation}
I_{z}(a,b)
=
\sum_{j=0}^{a-1} \binom{a+b-1}{j}
z^{a+b-1-j} (1-z)^j\label{eq:Iz_binom}
\end{equation}
for $a,b\in\IN$ with $a\geq b$.
With $z=1/2$, $a=m+\Delta$, and $b=m$, this implies
\begin{equation}
I_{1/2}(m+\Delta,m)
=
\frac{1}{2^{2m+\Delta-1}}
\sum_{j=0}^{m+\Delta-1} \binom{2m+\Delta-1}{j}
\enspace.
\end{equation}
The sum of the first binomial coefficients can be upper-bounded (e.g., \cite[Proof of Theorem~5.3.2]{lovasz03}) via
\begin{equation}
\sum_{j=0}^k \binom{n}{j}
\leq
2^{n-1} \exp\left( - \frac{(n-2k-2)^2}{4n-4k-4} \right)
\enspace.\label{eq:binom_exp}
\end{equation}
Setting $n = 2m+\Delta-1$ and $k = m+\Delta-1$ we get
\begin{equation}
I_{1/2}(m+\Delta,m)
\leq
\frac{1}{2} \exp\left( - \frac{(\Delta+1)^2}{4(m-1)} \right)
\leq
\exp\left(-\Omega\left(\frac{\Delta^2}{m}\right)\right)
\enspace.
\end{equation}
This concludes the proof of the lemma.
\end{proof}

Combining the above lemmas proves the second part of Theorem~\ref{thm:consensus}.

\section{Boolean Gates}\label{sec:boolean_gate}

In terms of circuit design, the A-B protocol can be viewed
  as a differential signal amplifier.
Differential signaling has applications in systems
  that require high resilience to noise, and thus an application
  for our inherently growing systems is natural.

In this section we study a protocol that allows to
  compute the logical \NAND\ of two signals,
  however with a loss of signal quality at the output.
The A-B protocol is then applied to regenerate the signal, obtaining a clear $0$ or $1$ with high probability.
Note that the \NAND\ gate protocol is easily generalized to arbitrary two-input Boolean functions, and so is its analysis.

We start with some notation.
A \emph{signal} is from a finite alphabet $\Sigma = \{X, Y, \dots\}$.
At each time $t \geq 0$, a signal $X \in \Sigma$ has a value $x(t) \in \{0,1,\bot\}$.
Following a technique from clockless circuit design \cite{spars2002principles,myers2001asynchronous} we encode
  the value of a signal as a dual-rail signal in the following way.
For each signal $X$, there are two species $X^0$ and $X^1$.
Intuitively, for $v \in \{0,1\}$, a large count of
  $X^v(t)$ and a low count of $X^{\neg v}(t)$ encodes for $x(t) = v$.
In fact, we will ask for a minimum gap in species counts between $X^{v}(t)$ and $X^{\neg v}(t)$.
If the signal is neither $0$ nor $1$, we will say that it has value $\bot$.
We will make the assumptions on the input signals precise in the sequel, and discuss guarantees
  on output signals when specifying the gate input/output behavior.

Let $X^0,X^1$ be species of a dual-rail encoding of
  signal $X$.
For convenience we write $X(t)$ for $X^0(t) + X^1(t)$.
For $n, \Delta \in \IN$, we say signal $X$ is \emph{initially $(n, \Delta)$-correct with value
  $x \in \{0,1\}$} if
\begin{align}
X(0) \geq n \enspace
\quad \text{ and } \quad
X^{\neg x}(0) \leq \frac{n-\Delta}{2}\enspace.
\label{eq:gatecorrect:assumption}
\end{align}
The \emph{initial gap} $X^{x}(0) - X^{\neg x}(0)$ of signal $X$ is thus bounded by
\begin{align}
X^{x}(0) - X^{\neg x}(0)=
    X^{x}(0) + X^{\neg x}(0) - 2X^{\neg x}(0) \geq \Delta\enspace.
\end{align}

\subsection{Dual-Rail \NAND\ Gate}
A dual-rail implementation of a \NAND\ gate with input signals $A,B$ and output signal $Y$
  is as a protocol with input species $\mathcal{I} = \{A^0, A^1, B^0, B^1\}$, output species
  $\mathcal{O} = \{Y^0, Y^1\}$, and no internal species.
Initial counts for outputs that are not inputs are $Y^0(0) = Y^1(0) = 0$.
Further, for all $a,b \in \{0,1\}$ and $y = \neg (a \wedge b)$, the protocol contains a reaction
\begin{equation*}
\ch{$A^a$ + $B^b$ ->[$\alpha$] $A^a$ + $B^b$ + $Y^y$}\enspace,
\end{equation*}
where $\alpha > 0$ is the gate's rate constant.
Since all species are permanently replicating, we further have
  the obligatory duplication reactions
  \ch{$A^i$ ->[$\gamma$] $2A^i$}, 
  \ch{$B^i$ ->[$\gamma$] $2B^i$}, and 
  \ch{$Y^i$ ->[$\gamma$] $2Y^i$}
  for $i\in\{0,1\}$.
\figref{fig:gate} depicts the \NAND\ gate with the subsequent amplification protocol.

\begin{figure}[h]
\centering
\tikzstyle{branch}=[fill,shape=circle,minimum size=3pt,inner sep=0pt]
\begin{tikzpicture}[%
  scale=1.5, every node/.style={scale=1.5}, 
  ampstyle/.style={draw=gray,dashed,rectangle,inner xsep=5pt, inner ysep=8pt}
]

\node[nand gate US, scale=1.6, draw, rotate=0] at (0,0)
  (nor2) {};
\node[ampstyle] at ($(nor2) + (2,0)$) 
  (amp) {\footnotesize amplify};

% output nand
\coordinate (o0) at ($ (nor2.output) + (0, 0.05) $);
\coordinate (o1) at ($ (nor2.output) + (0,-0.05) $);
  
\draw (o0) -- (o0-|amp.west) node[midway, above, yshift=-0.5pt]
  {\scriptsize $Y^0$};
\draw (o1) -- (o1-|amp.west) node[midway, below, yshift=0.5pt]
  {\scriptsize $Y^1$};  
  
% input A nand 
\coordinate (a0) at ($ (nor2.input 1) + (0, 0.2) $);
\coordinate (a1) at ($ (nor2.input 1) + (0, 0.1) $);  
  
\draw (a0) -- ++(-0.5,0) node[midway, left, xshift=-4pt, yshift=3pt]
  {\scriptsize $A^0$};
\draw (a1) -- ++(-0.5,0) node[midway, left, xshift=-4pt, yshift=-1pt]
  {\scriptsize $A^1$};  
  
% input B nand 
\coordinate (b0) at ($ (nor2.input 2) + (0,-0.1) $);
\coordinate (b1) at ($ (nor2.input 2) + (0,-0.2) $);  
  
\draw (b0) -- ++(-0.5,0) node[midway, left, xshift=-4pt, yshift=3pt]
  {\scriptsize $B^0$};
\draw (b1) -- ++(-0.5,0) node[midway, left, xshift=-4pt, yshift=-1pt]
  {\scriptsize $B^1$};  
  
% amp out
\coordinate (am0) at ($ (amp.east) + (0, 0.05) $);
\coordinate (am1) at ($ (amp.east) + (0,-0.05) $);  

\draw[gray] (am0) -- ++(0.5,0) node[midway, right, xshift=6pt, yshift=3pt]
  {\scriptsize $Z^0$};
\draw[gray] (am1) -- ++(0.5,0) node[midway, right, xshift=6pt, yshift=-1pt]
  {\scriptsize $Z^1$};  

\end{tikzpicture}
\caption{Dual-rail \NAND\ gate with input signals $A$ and $B$ an output signal $Y$. Successive amplification of $Y$ to signal $Z$ shown in gray.}
\label{fig:gate}
\end{figure}
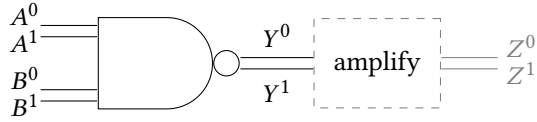

In Section~\ref{sec:gate_correct} we will show that the \NAND\ gate ensures the following input-output specification:

\begin{thm}\label{thm:gate}
Assume that the \NAND\ gate's input signals $A, B$
  are dual-rail encoded signals,
  and that they are
  initially $(n,\Delta)$-correct with values $a,b \in \{0,1\}$, respectively, where $n \in \IN^+$ and
  $\Delta \geq 0.62 \cdot \max\left\{A(0),B(0)\right\}$.
Then with high probability, there exists some time $t=O(1)$ such that $Y(t) = n$ and
$Y^y(t) - Y^{\neg y}(t) = \Omega(n)$ for the output signal~$Y$ where~$y = \neg (a \wedge b)$ is the correct \NAND\ output based on the initial values $a,b$ of signals~$A$ and~$B$, respectively.
\end{thm}

\subsection{Gate Correctness and Performance}
\label{sec:gate_correct}

We now turn to the proof of Theorem~\ref{thm:gate}.
For our analysis we need a bound on the regularized incomplete beta function~$I_{3/4}$.
\opt{append}{For space reasons, we postponed its proof to the appendix.}

\begin{lemrep}\label{lem:rate:decreasing}
For $X \geq Y$, it is
$I_{3/4}(X,Y) \leq \frac{1}{2} \exp\bigg(-\frac{(X-Y+1)^2}{4(Y-1)} + (X+Y-1)\log\frac{3}{2}\bigg)$.
In particular, for $m,\Delta \geq 0$,
\begin{align*}
I_{3/4}(m+\Delta,m)
\leq
\frac{1}{2}
\exp\left(-\frac{(\Delta+1)^2}{4(m-1)} + \frac{2m+\Delta}{2}\right)\enspace.
\end{align*}
\end{lemrep}
\begin{proof}
Instantiating \eqref{eq:Iz_binom} with $z=\frac{3}{4}$, $a= X$ and $b=Y$, we have
\begin{align*}
I_{3/4}(X,Y) & = \sum^{X-1}_{j=0} \binom{X+Y-1}{j} \bigg(\frac{3}{4}\bigg)^{X+Y-1-j}\bigg(\frac{1}{4}\bigg)^j\\
&=  \sum^{X-1}_{j=0} \binom{X+Y-1}{j} \bigg(\frac{1}{4}\bigg)^{X+Y-1} 3^{X+Y-1-j}\\
&= \bigg(\frac{1}{4}\bigg)^{X+Y-1}  \sum^{X-1}_{j=0} \binom{X+Y-1}{j}3^{X+Y-1-j}\\
&\leq 3^{X+Y-1} \bigg(\frac{1}{4}\bigg)^{X+Y-1}  \sum^{X-1}_{j=0} \binom{X+Y-1}{j}\\
\intertext{and from \eqref{eq:binom_exp} with $n=X+Y-1$ and $k=X-1$,}
&\leq  3^{X+Y-1} \bigg(\frac{1}{4}\bigg)^{X+Y-1}  \exp\bigg(-\frac{(X-Y+1)^2}{4(Y-1)}\bigg)\\
&\leq \frac{1}{2} \bigg(\frac{3}{2}\bigg)^{X+Y-1} \exp\bigg(-\frac{(X-Y+1)^2}{4(Y-1)}\bigg)\\
&=  \frac{1}{2} \exp\bigg(-\frac{(X-Y+1)^2}{4(Y-1)} + (X+Y-1)\log\frac{3}{2}\bigg)\enspace.
\end{align*}
The lemma's second inequality follows from
setting $X=m+\Delta$ and $Y=m$ in the above inequality.
Assuming that $m,\Delta \geq 0$, and noting that $\log\frac{3}{2} < \frac{1}{2}$,
  we obtain by algebraic manipulation
\begin{align}
-\frac{(X-Y+1)^2}{4(Y-1)} + (X+Y-1)\log\frac{3}{2}
&\leq
-\frac{(\Delta+1)^2}{4(m-1)} +  \frac{2m+\Delta}{2}\enspace;
\end{align}
from which the lemma follows.
\end{proof}

The following lemma shows that for $z=3/4$,
  the function $(x,y) \mapsto I_z(x,y)$ is non-decreasing in $(x,y)$ along
  the discretized line with slope $1/3$.
\opt{append}{Its proof is postponed to the appendix.}

\begin{lemrep}\label{lem:beta:increasing}
If $X \geq 3Y \geq 0$, then
$I_{3/4}(X,Y) \leq I_{3/4}(X+3,Y+1)$.
\end{lemrep}
\begin{proof}
We use the two identities,
\begin{align}
I_z(\alpha+1,\beta)
&=
I_z(\alpha,\beta) - \frac{z^\alpha(1-z)^\beta \Gamma(\alpha+\beta)}{\Gamma(\alpha+1)\Gamma(\beta)}\label{eq:beta:rec:a}\\
I_z(\alpha,\beta+1)
&=
I_z(\alpha,\beta) + \frac{z^\alpha(1-z)^\beta \Gamma(\alpha+\beta)}{\Gamma(\alpha)\Gamma(\beta+1)}\label{eq:beta:rec:b}
\enspace.
\end{align}
We have,
\begin{align*}
I_{3/4}(X+3,Y+1) - I_{3/4}(X,Y) &=  I_{3/4}(X+3,Y+1) - I_{3/4}(X+2,Y+1)\\
                                & \quad + I_{3/4}(X+2,Y+1) - I_{3/4}(X+1,Y+1)\\
                                & \quad  + I_{3/4}(X+1,Y+1) - I_{3/4}(X,Y+1)\\
                                & \quad  + I_{3/4}(X,Y+1) - I_{3/4}(X,Y)\enspace.
\end{align*}
Invoking~\eqref{eq:beta:rec:a} with $\alpha=X+2$ and $\beta=Y+1$ for the first term,
$\alpha=X+1$ and $\beta=Y+1$ for the second, and $\alpha=X$ and $\beta=Y+1$ for the third, as well as~\eqref{eq:beta:rec:b} with $\alpha=X$, $\beta=Y$ for the last term yields
\begin{align*}
I_{3/4}(X+3,Y+1) - I_{3/4}(X,Y)
& = -\frac{\left(\frac{3}{4}\right)^{X+2}\left(\frac{1}{4}\right)^{Y+1}\Gamma(X+Y+3)}{\Gamma(X+3)\Gamma(Y+1)}
    -\frac{\left(\frac{3}{4}\right)^{X+1}\left(\frac{1}{4}\right)^{Y+1}\Gamma(X+Y+2)}{\Gamma(X+2)\Gamma(Y+1)}\\
&\quad -\frac{\left(\frac{3}{4}\right)^{X}\left(\frac{1}{4}\right)^{Y+1}\Gamma(X+Y+1)}{\Gamma(X+1)\Gamma(Y+1)}
     +\frac{\left(\frac{3}{4}\right)^{X}\left(\frac{1}{4}\right)^{Y}\Gamma(X+Y)}{\Gamma(X)\Gamma(Y+1)}\\
&= \frac{\left(\frac{3}{4}\right)^X\left(\frac{1}{4}\right)^Y\Gamma(X+Y)}{\Gamma(X+3)\Gamma(Y+1)}\cdot \Bigg( X(X+1)(X+2)
 - \left(\frac{3}{4}\right)^2\frac{1}{4}(X+Y)(X+Y+1)(X+Y+2)\\
&\quad\quad - \frac{3}{4}\cdot\frac{1}{4}(X+Y)(X+Y+1)(X+2)
 - \frac{1}{4}(X+Y)(X+1)(X+2) \Bigg)\\
&\stackrel{Y \le X/3}{\geq} \frac{\left(\frac{3}{4}\right)^X\left(\frac{1}{4}\right)^Y\Gamma(X+Y)}{\Gamma(X+3)\Gamma(Y+1)} \cdot \frac{1}{24} (8X + 11) X\enspace,
\end{align*}
which is nonnegative since $X \geq 0$.
\end{proof}

We are now in the position to show a lower bound on the
  probability for a discrete time Yule process with two species $X$ and $Y$,
  that $\displaystyle\lim_{k \to \infty } X_k/(X_k+Y_k) < 3/4$,
  given that the initial values fulfill $X_0/(X_0+Y_0) > 3/4$ and
  that there is a step~$\ell$ with $X_\ell/(X_\ell+Y_\ell) \leq 3/4$.

\begin{lemrep}\label{lem:omega}
Let $X$ and $Y$ be species from a Yule process.
Assume that $X_0/(X_0+Y_0) > 3/4$ for the initial values.
Then
\begin{align*}
\mathbb{P}\left(\displaystyle \lim _{k\to \infty} \frac{X_k}{X_k+Y_k} < \frac{3}{4}
\,\,\Big|\,\, 
\exists \ell : \frac{X_\ell}{X_\ell + Y_\ell} \leq \frac{3}{4} \right)
\geq \omega\bigl(X_0, Y_0\bigr)
\intertext{where}
\omega\bigl(X_0, Y_0\bigr) =
\inf\left\{
I_{3/4}(x,y) \,\big|\,
x \ge X_0 \wedge y \ge Y_0 +1 \wedge x \in 3y - \{0,1,2\}
\, \right\} \enspace.
\end{align*}
Moreover, $\omega\bigl(X_0, Y_0\bigr) > 0.444$
\end{lemrep}
\begin{proof}
By assumption $X_0/(X_0+Y_0) > 3/4$.
Let $\ell \ge 1$ be minimal such that $X_\ell/(X_\ell+Y_\ell) \leq 3/4$.
By assumption such an $\ell$ exists.
By minimality of $\ell$, we have
\begin{align}
X_\ell &\leq 3Y_\ell \quad\text{ and }\quad
X_{\ell-1} > 3Y_{\ell-1}\enspace.
\end{align}
From the fact that $X,Y$ follow a Yule process, this can only
be the case if $Y$ has increased from step $\ell-1$ to $\ell$, i.e.,
\begin{align}
X_\ell &= X_{\ell-1} \geq X_0\label{eq:omega:x}\\
Y_{\ell} &= Y_{\ell-1} + 1 \geq Y_0 + 1\enspace.\label{eq:omega:y}
\end{align}
Thus, $X_\ell > 3Y_\ell-3$ from which $X_\ell \geq 3Y_\ell-3$ and further,
\begin{equation}
X_\ell \in 3Y_\ell - \{0,1,2\} \enspace.\label{eq:omega:13}
\end{equation}
For a Yule process with species $X'$ and $Y'$, and arbitrary initial counts $X'_0 = x$ and $Y'_0 = y$, we have
\begin{equation}
\mathbb{P}\left(\displaystyle \lim _{k\to \infty} \frac{X'_k}{X'_k+Y'_k}
< \frac{3}{4}\right) = I_{3/4}(x,y)\enspace.\label{eq:omega:I}
\end{equation}
The first inequality of the lemma now follows from
 \eqref{eq:omega:x}, \eqref{eq:omega:y}, \eqref{eq:omega:13}, and \eqref{eq:omega:I}.

We next show the second inequality of the lemma.
For that purpose, we remark that any $(x,y)$ in
\begin{align*}
 S= \left\{ x \ge X_0 \wedge y \ge Y_0 +1 \wedge x \in 3y - \{0,1,2\} \right\} 
\end{align*}
with $X_0 \geq 1$ and $Y_0 \geq 1$ is of the form
\begin{align}
s_0 + m \cdot (3,1) \quad \text{ where } s_0 \in \{(4,2), (5,2), (6,2)\} \text{ and }m \in \IN\enspace.\label{eq:omega:s}
\end{align}

Assume $x = 3y - \Delta$ with $\Delta \in \{0,1,2\}$.
Choosing $s_0 = (6-\Delta,2)$ and $m = y-2 \geq 0$, and applying
\eqref{eq:omega:s} yields
\begin{align*}
(6-\Delta,2) + (y-2) \cdot (3,1) = (3y-\Delta,y) = (x,y)\enspace,
\end{align*}
from which the claim follows.

By repeatedly applying Lemma~\ref{lem:beta:increasing} to an element $(x,y)$ in $S$,
  we have from \eqref{eq:omega:s} that
\begin{align}
\omega(X_0,Y_0) \geq \min\{ I_{3/4}(4,2), I_{3/4}(5,2), I_{3/4}(6,2) \} = I_{3/4}(6,2) > 0.444 \enspace.
\end{align}
The lemma follows.
\end{proof}

Making use of Lemma~\ref{lem:omega}, we next prove an upper bound on the probability that the two-species discrete-time Yule process~$X,Y$, with an initial large majority of $X$, eventually hits a step where its relative population size drops
to $\frac{3}{4}$ or below.
\opt{append}{Its proof is similar to that of Lemma~\ref{lem:collision:yule} and is postponed to the appendix.}

\begin{lemrep}\label{lem:prob:X}
Let $X$ and $Y$ be species from a Yule process.
Assume that $\frac{X_0}{X_0+Y_0} > \frac{3}{4}$.
Then
\begin{equation}
\mathbb{P}\left(\exists k :\frac{X_k}{X_k+Y_k} \leq \frac{3}{4}\right) < \frac{I_{3/4}\left(X_0,Y_0\right)}{0.444}\enspace.
\end{equation}
\end{lemrep}
\begin{proof}
By assumption $X_0 > 3Y_0$.
Further, we have
\begin{eqnarray}
I_{3/4}(X_0, Y_0)&=&\mathbb{P}\left( \displaystyle\lim_{k\to \infty} \frac{X_k}{X_k+Y_k} < \frac{3}{4} \right)\\\nonumber
&=&\mathbb{P}\left( \displaystyle \lim _{k\to \infty} \frac{X_k}{X_k+Y_k} < \frac{3}{4} \Bigm| \exists k: \frac{X_k}{X_k+Y_k} \leq\frac{3}{4}\right) \cdot \mathbb{P}\left(\exists k : \frac{X_k}{X_k+Y_k} \leq\frac{3}{4}\right) \\\nonumber
& & +\, \mathbb{P}\left(\displaystyle \lim _{k\to \infty} \frac{X_k}{X_k+Y_k} < \frac{3}{4} \wedge \forall k: \frac{X_k}{X_k+Y_k} > \frac{3}{4}\right)\\\nonumber 
 &\stackrel{\text{Lemma}~\ref{lem:omega}}{>}& 0.444 \cdot \mathbb{P}\left(\exists k :\frac{X_k}{X_k+Y_k} \leq \frac{3}{4}\right)\enspace.
\end{eqnarray}
The lemma follows.
\end{proof}

The following lemma provides a lower bound on the probability
 that the dual-rail encoding of signals $A$ and $B$, that are both initially $(n, \Delta)$-correct,
 for $\Delta > n/2$, remains separated as their species grow.
 
\begin{lem}\label{lem:gatecorrect}
Let $A^0,A^1$ as well as $B^0,B^1$ be species of a dual-rail encoding of signals $A$ and $B$.
Assume that each species follows a Yule processes.
If signals $A$ and $B$ are initially $(n, \Delta)$-correct with $n, \Delta \in \IN$ with $\Delta > \frac{n}{2}$ for some $a,b\in \{0,1\}$,
  then
\begin{align}
\mathbb{P}\left(
\forall t \ge 0 :
\frac{A^a(t)}{A(t)} > \frac{3}{4} \,\wedge\,
\frac{B^b(t)}{B(t)} > \frac{3}{4}
\right) \geq
\left(1-\frac{1}{2\cdot 0.444} \exp\bigg(
\frac{1}{2}\left(-\frac{\Delta^2}{(n-\Delta)} + \max\{A(0),B(0)\}\right)
\bigg)\right)^2\enspace.
\label{eq:gatecorrect:bound}
\end{align}
\end{lem}
\begin{proof}
By Independence of the two Yule processes, we have 
\begin{align}
\mathbb{P}\Biggl(
\forall t \ge 0 :
\frac{A^a(t)}{A(t)} > \frac{3}{4} \,\wedge\,
\frac{B^b(t)}{B(t)} > \frac{3}{4}
\Biggr) &= 
\mathbb{P}\Biggl(
\forall t \ge 0 : \frac{A^a(t)}{A(t)} > \frac{3}{4}
\Biggr) \cdot
\mathbb{P}\Biggl(
\forall t \ge 0 : \frac{B^b(t)}{B(t)} > \frac{3}{4}
\Biggr)\enspace.\label{eq:gatecorrect:34}
\end{align}
Further, since $A$ is $(n,\Delta)$-correct with $\Delta > \frac{n}{2}$,
\begin{align*}
A(0) &= 2A^a(0) - ( A^{a}(0) - A^{\neg a}(0)) \quad\Rightarrow\\
A^a(0) &\geq \frac{A(0) + \Delta}{2} \quad\Rightarrow\\
\frac{A^a(0)}{A(0)} &\geq \frac{n+\Delta}{2n} > \frac{3}{4}\enspace.
\end{align*}
By analogous arguments, $\frac{B^b(0)}{B(0)} > \frac{3}{4}$.
We may thus apply Lemma~\ref{lem:prob:X} twice to \eqref{eq:gatecorrect:34}, obtaining
\begin{align*}
\mathbb{P}\Biggl(
\forall t \ge 0 :
\frac{A^a(t)}{A(t)} > \frac{3}{4} \,\wedge\,
\frac{B^b(t)}{B(t)} > \frac{3}{4}
\Biggr)
&>
\Biggl(1- \frac{I_{3/4}\bigl(A^a(0),A^{\neg a}(0)\bigr)}{0.444}\Biggr)\cdot
\Biggl(1- \frac{I_{3/4}\bigl(B^b(0),B^{\neg b}(0)\bigr)}{0.444}\Biggr)\enspace.\label{eq:gatecorrect:final}
\end{align*}
We can now apply Lemma~\ref{lem:rate:decreasing} twice: for $X=A^a(0)$ and $Y=A^{\neg a}(0)$, and for $X=B^b(0)$ and $Y=B^{\neg b}(0)$.
For $A^a$ and $A^{\neg a}$, we further have
\begin{align*}
-\frac{\left(A^{a}(0)-A^{\neg a }(0)+1\right)^2}{4(A^{\neg a}(0)-1)} + \frac{A(0)}{2}
\leq
-\frac{\left(A^{a}(0)-A^{\neg a }(0)\right)^2}{4A^{\neg a}(0)} + \frac{A(0)}{2}
\stackrel{\eqref{eq:gatecorrect:assumption}}{\leq}
-\frac{\Delta^2}{4\frac{n-\Delta}{2}} + \frac{A(0)}{2}
=
\frac{1}{2}\left(-\frac{\Delta^2}{(n-\Delta)} + A(0)\right)\enspace.
\end{align*}
By analogous arguments for $B^b$ and $B^{\neg b}$, 
  the bound in \eqref{eq:gatecorrect:bound} follows.
\end{proof}

We next show in Lemma~\ref{lem:Delta_adaptedY} that when the \NAND\ gates has produced $n$ output species $Y^0$ and $Y^1$, a certain gap $\Delta > 0$ is guaranteed
  with a probability that depends on $n$ and $\Delta$.
However, instead of showing this for the original \NAND\ gate, we first prove that the bound holds for an adapted version where $Y^0$ and $Y^1$ do not duplicate.
We later extend the result to the original \NAND\ gate in Lemma~\ref{lem:inputsoutputs_gate}.
\opt{append}{The proofs of Lemmas~\ref{lem:Delta_adaptedY} and \ref{lem:inputsoutputs_gate} are given in the appendix.}

\begin{lemrep}\label{lem:Delta_adaptedY}
Consider an adapted version of the \NAND\ gate with dual-rail encoded input signals $A,B$ and output signal $Y$.
In the adapted version, species $Y^0$ and $Y^1$ do not duplicate.
Further, assume that for some $a,b \in \{0,1\}$,
\begin{align}
\forall t \ge 0 :
\frac{A^a(t)}{A(t)} > \frac{3}{4} \,\wedge\,
\frac{B^b(t)}{B(t)} > \frac{3}{4}\enspace.
\end{align}
Then, with $y = \neg(a \wedge b)$ being the correct Boolean output
  of the gate, for any $t \geq 0$ and $\Delta,n \in \IN$ with $\Delta \leq n/8$,
\begin{align}
 \mathbb{P}\left(Y^y(t) - Y^{\neg y}(t) > \Delta \bigm| Y(t) = n\right) \geq
 1 -
 \exp\left(-\frac{\left(\frac{n}{8}-\Delta\right)^2}{2n}\right)\notag\enspace.
\end{align}
\end{lemrep}
\begin{proof}
From the assumption on the inputs, we have that the probability of the \NAND\ gate to chose species $A^a$ and $B^b$ when producing an output species,
 is at least $p = \left(\frac{3}{4}\right)^2$.
Likewise a wrong output is produced with probability at most $1-p$.

Consider the discrete random walk on $\IZ$, starting at position $D_0 = 0$, and at step $i \geq 1$, incrementing $D_{i-1}$ by one with probability $p$, and decrementing by one with probability $1-p$.
It is easy to construct a coupling such that $D_n \leq Y^y(t) - Y^{\neg y}(t)$, given that $Y(t) = n$.

Let $I_i$, $i \geq 1$, be a
sequence of i.i.d.\ Bernoulli trials with success probability $p$,
and $R_n = \sum_{i=1}^n I_i$.
Then $R_n$ follows a Binomial distribution and $2R_n-n$ is identically distributed to $D_n$.
Thus,
\begin{align}
\mathbb{P}\left(D_n > \Delta\right) & =
1 - \mathbb{P}\left(R_n \leq \frac{\Delta+n}{2}\right)
= 1 - \sum_{i=0}^{\frac{\Delta+n}{2}} {n \choose i}\, p^i(1-p)^{n-i}\enspace.
\end{align}
Applying Hoeffding's inequality \cite{hoeffding1994probability} for sums of Bernoulli trials, we obtain for
  $\Delta \leq (2p-1)n = \frac{n}{8}$,
\begin{align}
\sum_{i=0}^{k} {n \choose i}\, p^i(1-p)^{n-i} &\leq
\exp\left(-2\frac{\left(np-k\right)^2}{n}\right)
\end{align}
where $k = \frac{\Delta+n}{2}$.
Thus,
\begin{align}
\mathbb{P}\left(D_n > \Delta\right)
&\geq 1 - \exp\left(-2\frac{\left(np-\frac{\Delta+n}{2}\right)^2}{n}\right)
= 1 - \exp\left(-\frac{\left(\frac{n}{8}-\Delta\right)^2}{2n}\right)\enspace.
\end{align}
The lemma follows.
\end{proof}

\begin{lemrep}\label{lem:inputsoutputs_gate}
Consider the \NAND\ gate with dual-rail encoded input signals $A,B$ and output signal $Y$.
If for some $a,b \in \{0,1\}$,
\begin{align}
\forall t \ge 0 :
\frac{A^a(t)}{A(t)} > \frac{3}{4} \,\wedge\,
\frac{B^b(t)}{B(t)} > \frac{3}{4}\enspace,
\end{align}
$A(0) \geq n$, and $B(0) \geq n$
then, letting $y = \neg (a \wedge b)$ be the correct Boolean output
  of the gate, with high probability there exists a $t = O(1)$ such that
 $Y^y(t) - Y^{\neg y}(t) = \Omega(n)$ and
 $Y(t) = n$.
\end{lemrep}
\begin{proof}
Consider the variant of the \NAND\ gate from Lemma~\ref{lem:Delta_adaptedY} where $Y^0$ and $Y^1$ do not duplicate.
Let $T > 0$ be the earliest time $t$ when $Y(t) = n$.

By assumption, for all $t'\geq 0$, $A(t') \geq n$ and $B(t') \geq n$.
Thus the gate's production rate of $Y$ species is at least $n^2\alpha$.
It follows that with high probability $T \leq \frac{\log n}{\alpha n}$.

We will next upper bound the count of species~$Y$  that would have been produced if duplication were in place during time $[0,T]$.
For that purpose, assume that all $Y$ species generated
  by the gate during $[0,T]$ are already produced at time~$0$.
Then, the count of species~$Y$ generated by duplication,
  let us call them $\hat{Y}$, follows a single species Yule process with initial count $\hat{Y}(0) = n$.
Thus, $\hat{Y}(T)$ follows a negative binomial distribution with
 success probability $p = 1 - e^{-\gamma T}$ and $r = \hat{Y}(0)$, i.e.,
\begin{align}
\IP\left(\hat{Y}(T) = k \right) &=
{k-1 \choose \hat{Y}(0) - 1}\,
e^{-\gamma T \hat{Y}(0)}
\left(1-e^{-\gamma T}\right)^{k - \hat{Y}(0)} \enspace.
\end{align}
Further, for the expected count of species
  generated by duplication, minus the initial
  $\hat{Y}(0)$ that were generated by the gate,
we have,
\begin{align}
\IE\left(\hat{Y}(T) - \hat{Y}(0)\right) =
\frac{r}{1-p}-r = 
\hat{Y}(0)(e^{\gamma T}-1) \leq
n \left(e^{\frac{\gamma}{\alpha}\frac{\log n}{n}}-1\right)
\enspace.\label{eq:E_exp}
\end{align}

We next show that,
\begin{align}
\IE\left(\hat{Y}(T) - \hat{Y}(0)\right) = O(\log n)\enspace.\label{eq:E_Olog}
\end{align}
Setting
  $g = \gamma/\alpha$, and letting
  $C = g e^g$, Equation~\eqref{eq:E_Olog} follows from the fact that for all $n >0$,
\begin{align*}
n \left(e^{\frac{g\log n}{n}}
-1\right)
&\leq 
C \log n
\quad\quad\Leftrightarrow\\
e^{\frac{g\log n}{n}}
&\leq
C \frac{\log n}{n} + 1\enspace.
\end{align*}
Substituting $z = \log n/n$,
  the latter follows from
\begin{align}
\forall z\in[0,1] \colon\quad
e^{g z}
&\leq 
C z + 1
\enspace.\label{eq:Cz}
\end{align}
Inequality \eqref{eq:Cz}, follows by observing that
  it holds for $z = 0$, and that, by taking the
  $z$-derivative on both sides, we obtain
\begin{align*}
g e^{gz} \leq C\enspace.
\end{align*}
which is true for $z \in [0,1]$ by choice of $C = g e^g$; Equation~\eqref{eq:E_Olog} follows.

Noting that the variance $\sigma^2 = \Var\left(\hat{Y}(T)-\hat{Y}(0)\right)$ of a negative binomial distribution, with $r$ and $p$ as above, is
\begin{align*}
    \sigma^2 = \frac{\IE\left(\hat{Y}(T)-\hat{Y}(0)\right)}{1-p}\enspace,
\end{align*}
and setting $\mu = \IE(\hat{Y}(T)-\hat{Y}(0))$, we next apply Chebyshev's bound $\IP\left( |X -\mu |\geq \epsilon \right)\leq \frac{\sigma^2}{\epsilon^2}$.

In particular, the fact that with high probability less than $\mu+\epsilon$ species of $Y$ are generated by duplication, follows from
\begin{align}
\IP\left(\hat{Y}(T)-\hat{Y}(0)\geq \mu + \epsilon\right) \leq \frac{\sigma^2} {\epsilon^2} \leq
1/n\enspace.\label{eq:whp_eps}
\end{align}
Solving for $\varepsilon$ gives,
\begin{align}
\frac{\sigma^2} {\epsilon^2} &\leq
1/n\quad\Leftrightarrow\quad
\epsilon \geq \sigma\sqrt{n}\quad\Leftrightarrow\quad
\epsilon \geq
\sqrt{n \IE\left(\hat{Y}(T)-\hat{Y}(0)\right) e^{\gamma T}}\enspace.
\end{align}
Further, observing that $e^{\gamma T} = e^{\frac{g\log n}{n}} = O(1)$, and using \eqref{eq:E_Olog}, we obtain the existence of a function $h(\cdot)$, such that, if we choose
\begin{align}
\epsilon \geq h(n) =
O\left(\sqrt{n\log n}\right)\enspace,
\label{eq:h}
\end{align}
Inequality~\eqref{eq:whp_eps} is fulfilled.

Thus, together with \eqref{eq:E_Olog}, one obtains that with high probability $\hat{Y}(T) - \hat{Y}(0)$ is at most
\begin{align}
    \IE\left(\hat{Y}(T) - \hat{Y}(0)\right) + \varepsilon = O\left(\sqrt{n\,\log n}\right)\enspace.
    \label{eq:whp_Yhat}
\end{align}

Applying Lemma~\ref{lem:Delta_adaptedY} for $Y(T) = n$, we obtain a bound on the gap $\Delta = Y^y(T) - Y^{\neg y}(T)$, excluding those generated by duplication, that holds with high probability.
Choosing 
\begin{align*}
\Delta =
\frac{n}{8} - \sqrt{2n\,
\log n}\enspace,
\end{align*}
we apply Lemma~\ref{lem:Delta_adaptedY} for $n$ and $\Delta \leq \frac{n}{8}$, and obtain
\begin{align}
 \mathbb{P}\left(Y^y(t) - Y^{\neg y}(t) > \Delta \bigm| Y(t) = n\right) \geq
 1 -
 \exp\left(-\frac{\left(\frac{n}{8}-\Delta\right)^2}{2n}\right)\label{eq:YyY1y}\enspace.
\end{align}
By choice of $\Delta$,
\begin{align}
\Delta &\leq
\frac{n}{8} - \sqrt{2n\,
\log n}\quad\quad\Rightarrow\\
\left(\frac{n}{8}-\Delta\right)^2
& \geq 2n\log n\quad\quad\Rightarrow\\
\exp\left(
-\frac{\left(\frac{n}{8}-\Delta\right)^2}{2n}
\right)
& \geq
\frac{1}{n}\quad\quad\Rightarrow\\
1 -  \exp\left(
-\frac{\left(\frac{n}{8}-\Delta\right)^2}{2n}
\right)
& \geq
1 - \frac{1}{n}
\end{align}
Together with \eqref{eq:YyY1y}, we have
\begin{align}
\mathbb{P}\left(Y^y(t) - Y^{\neg y}(t) > \Delta \bigm| Y(t) = n\right) \geq
 1 - \frac{1}{n}\enspace.
\end{align}
Additionally accounting for the
  $Y$ species that have been generated by duplication until time $T$, by using \eqref{eq:whp_Yhat},
  we obtain that the gap $Y^y(T) - Y^{\neg y}(T)$ between correct output species $Y^y$ and incorrect output species $Y^{\neg y}$ at time
  $T$ in a gate with duplication, with high probability, fulfills
\begin{align*}
Y^y(T) - Y^{\neg y}(T)
\geq
\Delta - \left(\hat{Y}(T) - \hat{Y}(0)\right)
= \Omega\left(n\right)\enspace.
\end{align*}
The lemma follows.
\end{proof}

We are now in the position to prove Theorem~\ref{thm:gate},
  showing the correctness of the \NAND\ gate if
  each of the two dual-rail input signals has a sufficiently
  large gap between its rails.

% -------------- thm proof --------------------------
\begin{proof}[Proof of Theorem~\ref{thm:gate}.]
The theorem follows from Lemma~\ref{lem:inputsoutputs_gate}
  if its assumption holds with high probability.
The latter follows from Lemma~\ref{lem:gatecorrect} if
  the exponent $\frac{1}{2}\left(-\frac{\Delta^2}{(n-\Delta)} + \max\{A(0),B(0)\}\right)$ is in $\Omega\left(-\max\{A(0),B(0)\}\right)$.
We next show that this is the case.

Let $M = \max\{A(0),B(0)\}$.
From $\Delta \geq \mu M$ with $\mu = 0.62$ we have,
\begin{align*}
-\frac{\Delta^2}{n-\Delta} + M
\leq
-\frac{\mu^2 M^2}{M - \mu M} + M
\leq
M\left(1-\frac{\mu^2}{(1-\mu)}\right)\enspace.
\end{align*}
It thus remains to show that $\left(1-\frac{\mu^2}{(1-\mu)}\right) < 0$.
By algebraic manipulation, this is the case if
  $\mu \in \left(\frac{1}{2} (\sqrt{5} - 1), 1\right)$,
  which is true by assumption.
The theorem follows.
\end{proof}

\section{Simulations}\label{sec:simulations}

\subsection{A-B Protocol Simulations}

Simulations corresponding to the A-B protocol complement the theoretical results above. The A-B protocol is simulated in \figref{AB_diff}
for the probability that species A survives, while species B goes extinct. The birth and death rates, $\gamma$ and $\delta$, are both set to 1. The probability that the protocol converges on A is primarily dependent on the difference in initial population size $A_0 - B_0$. Larger populations are only slightly less sensitive to the difference: \figref{AB_diff} demonstrates that the total population size across two orders of magnitude has a small effect compared to the difference between species.
Indeed, this behavior qualitatively matches the bound in
  Theorem~\ref{thm:consensus} with $-\Omega(\Delta^2/n)$
  in the exponent.

\begin{figure}
 \centering
 \includegraphics[scale=.35]{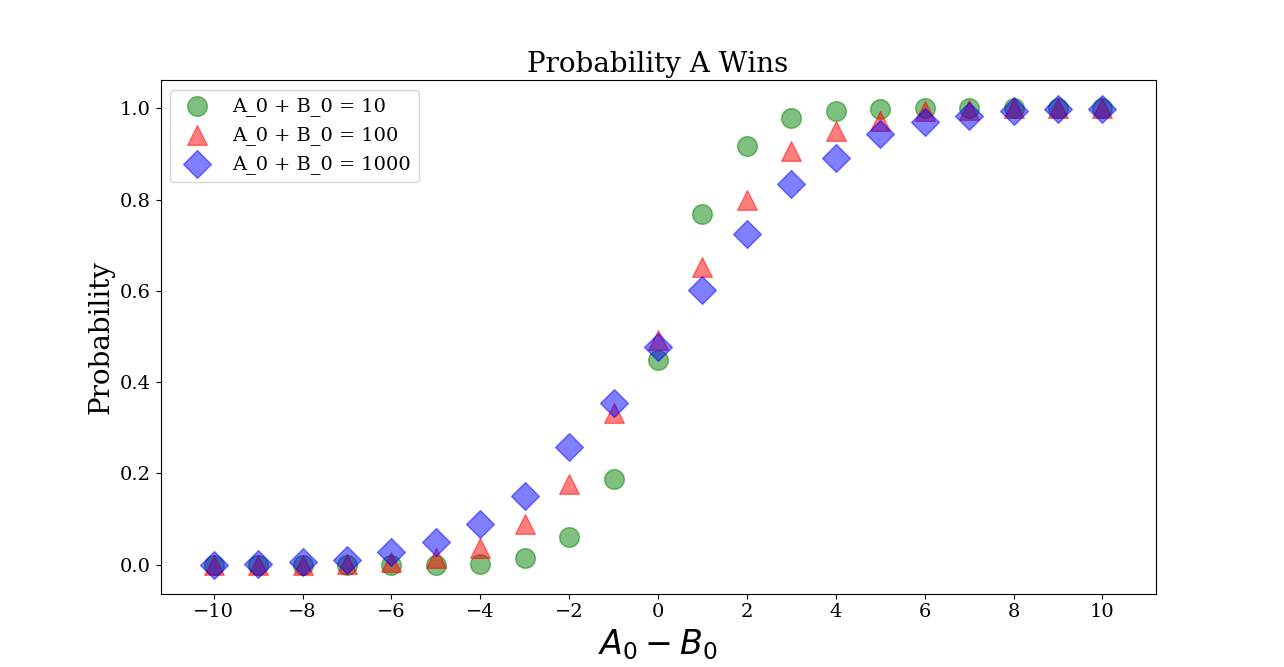}
\caption{The probability that species A survives while species B goes extinct is sharply dependent on their initial difference in population count $A_0 - B_0$. The sharpness of the transition is inversely proportional to initial population size $A_0 + B_0$.}
\label{AB_diff}
\end{figure}

The dependence of expected convergence time for the A-B protocol is explored over its reaction rate constants and initial conditions in \figref{AB_sim}.
Exponential changes in rate constants yield exponential changes in convergence time.
As expected, the convergence time is more strongly dependent on the death rate constant~$\delta$, than the birth rate constant~$\gamma$.
Convergence time sharply increases if the initial concentrations of the two species $A$ and $B$ are proportional.
The off-diagonal initial concentrations converge faster for larger population sizes since the absolute difference in concentrations is larger.

\begin{figure}
\begin{subfigure}{.5\textwidth}
  \centering
  \includegraphics[width=0.9\linewidth]{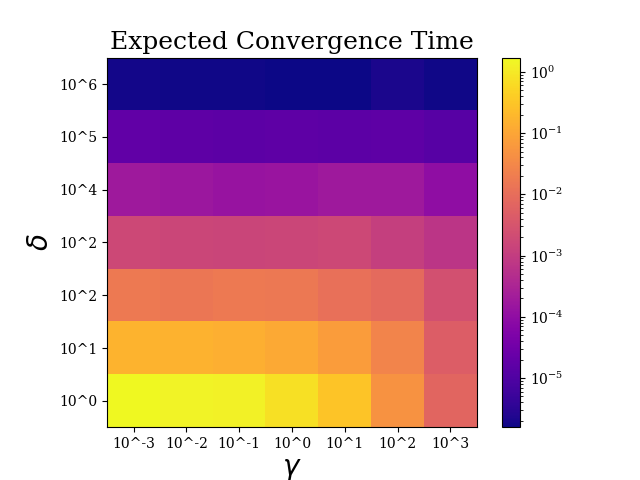}
\end{subfigure}\hfill
\begin{subfigure}{.5\textwidth}
  \centering
  \includegraphics[width=0.9\linewidth]{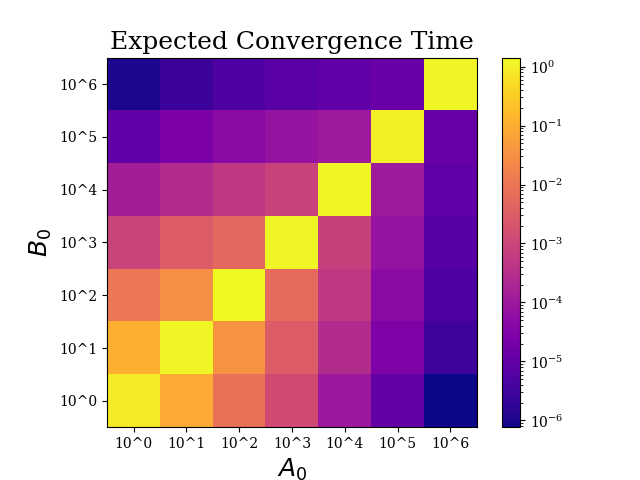}
\end{subfigure}
\caption{Log-scaled expected convergence time of the A-B protocol is represented by color. Corresponding values are shown on the adjacent vertical bar. \textbf{Left:} rate constants $\gamma$ and $\delta$ with $A_0=B_0=100$. \textbf{Right:} initial populations sizes with $\gamma=0.01$ and $\delta=1$.}
\label{AB_sim}
\end{figure}

\subsection{\emph{In silico} Biological Implementation}
While the studied model is a simplification, it represents core functions that constitute collective decision-making among biological species, and is readily adaptable for specific biological applications.
If reactions are modified such that one of the two reactants does not change, the model could represent one-way messaging equivalent to a conjugation event between a sender and receiver bacterial cell~\cite{marken2019addressable}.
Similarly, if the messages $A$ and $B$ are coded as free species diffusible between senders and receivers, it could represent communication between bacterial cells using bacteriophage particles as messages~\cite{ortiz2012engineered}.

In this section, we discuss a plausible biological implementation with \emph{E. coli} bacteria that use conjugation to communicate.
Conjugation is a method of genetic communication in which circular DNA plasmids are transferred from a sender cell to a receiver cell.
An F plasmid allows a cell to be a sender during conjugation.
The receiver can be engineered to express a logical function using the received plasmid and its existing DNA, although the internal implementation is not detailed for this simulation.
A conjugation reaction with a sender~$S$ and a receiver~$R$ is described by \ch{$R$ + $S$ ->[$\delta$] $f(R,S)$ + $S$}, where $\delta$ is the conjugation rate constant.
Both, the amplifier and the \NAND\ gate follow this scheme.
For the amplifier, $f(R,S)=\emptyset$ and for the \NAND\ gate $f(R,S)=Y$, where $Y$ is the gate's corresponding output species.
While with wild-type F plasmids, \emph{E. coli} are either senders (with F plasmid) or receivers (without F plasmid), there exist engineered systems that allow the same cell (with F plasmid) to be both a sender and a receiver~\cite{DimitriuLBMBLT14,marken2019addressable}.
Note that a single cell still cannot act as both the sender and the receiver during a single reaction.

The growth of the \emph{E. coli} is modeled by a logistic model with a carrying capacity of $10^9$ cells.
Reaction rate constants for duplication $\gamma = 0.016$ and for conjugation $\delta = 10^{-11}$ have been taken from Dimitriu \emph{et al.}~\cite{DimitriuLBMBLT14}. 
For our implementation, amplification of the gate's inputs and outputs was executed in parallel to the gate's protocol.
The simulations discussed in the following suggest that sequential execution is not required for correctness and performance, greatly simplifying the biological design.
If all possible gate reactions were used, inputs that lead to $Y^1$ would be more susceptible to noise since there are more possible input pairs leading to $Y^1$ than $Y^0$ in a \NAND\ gate.
This was alleviated by selecting a subset of all possible gate reactions in which three reactions lead to $Y^1$~(see (1)--(3) below) and two reactions lead to $Y^0$~(see (4)--(5) below). 

\begin{multicols}{2}
\begin{enumerate}
    \item \ch{$A^1+B^0$ -> $A^1+Y^0$}
    \item \ch{$A^0+B^1$ -> $A^0+Y^0$}
    \item \ch{$A^1+B^1$ -> $A^1+Y^0$}
    \item \ch{$A^0+B^0$ -> $A^0+Y^1$}
    \item \ch{$A^0+B^0$ -> $Y^1+B^0$}
\end{enumerate}
\end{multicols}

Simulation of our system for the four possible input choices are  shown in \figref{fig:NAND_sim}.
For performance with many individuals, simulations are done using the $\tau$-leaping approximation of stochastic simulation, in which multiple reactions occur during a dynamic time interval of $\tau$, before updating reaction rates~\cite{gillespie2001approximate, hoops2006copasi}. 
%Tau is dynamically selected such that the change in the reaction rates over tau remains small. 
The initial population size is set to $5 \times 10^8$, the carrying capacity to $1 \times 10^9$, and the initial input error to $10$\% of wrong input species per input.
Despite the low rate of communication from conjugation, the correct output species rapidly out-competes the incorrect output species for all input choices.

\begin{figure}
\centering
\includegraphics[width=1\linewidth]{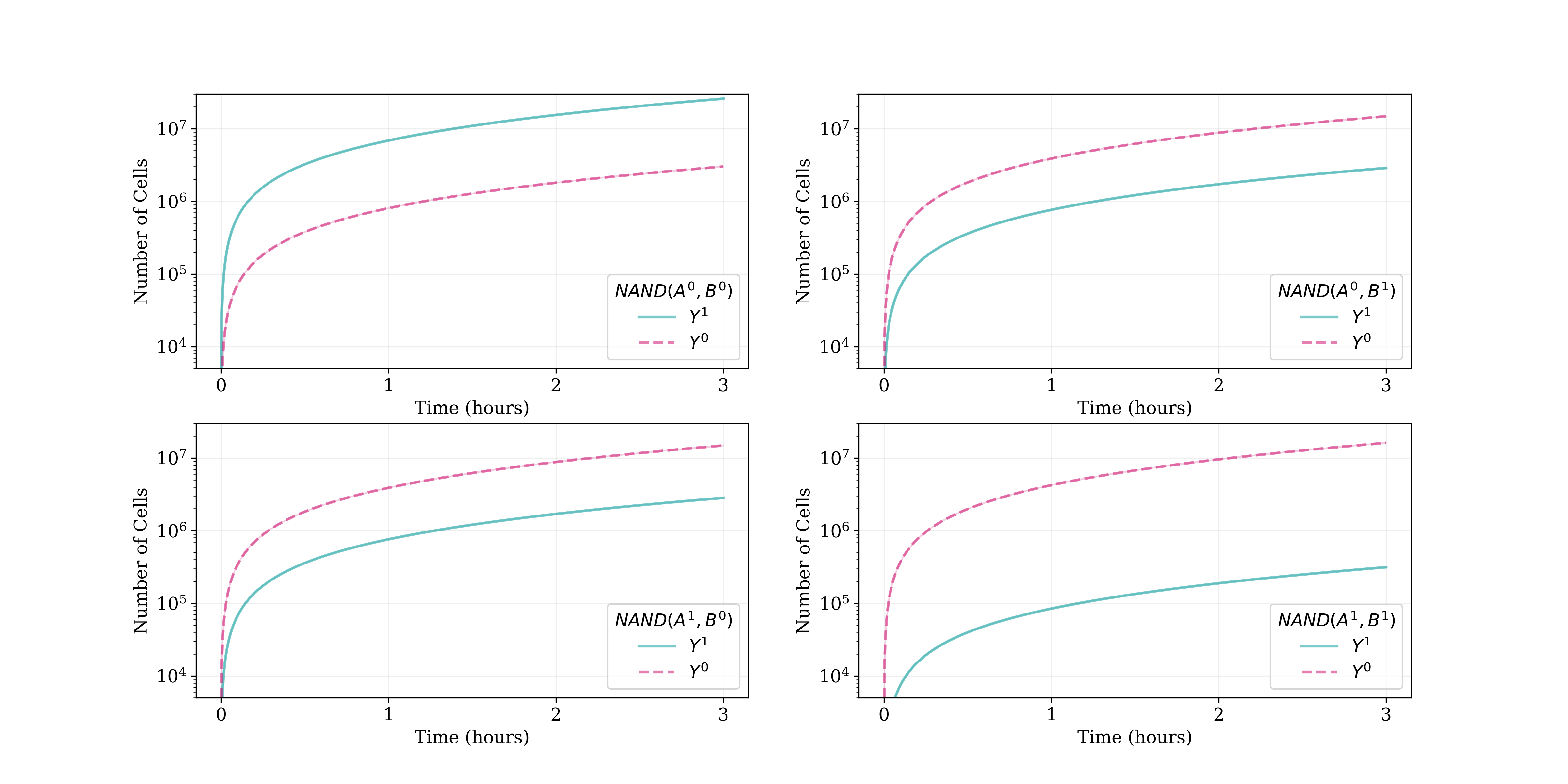}
\caption{A biologically plausible implementation of a \NAND\ gate with amplifiers on inputs and outputs.
Initial population size is $5 \times 10^8$, carrying capacity of $10^9$ cells.
Reaction rate constants were set to $\gamma= 0.016$ and $\delta = 10^{-11}$ \cite{DimitriuLBMBLT14}.
The output species is shown for each choice of inputs.
The initial input error is $10\%$.
All choices lead to correct, clearly separable outputs within half an hour.
Confidence intervals from 30 sample simulations are smaller than the width of the lines.}
\label{fig:NAND_sim}
\end{figure}

\section{Conclusions}\label{sec:conclusion}

We considered the majority consensus problem with continuous population growth in a stochastic setting, and established the A-B protocol between two competing species~$A$ and $B$ with birth reactions \ch{$A$ -> $2A$} and \ch{$B$ -> $2B$}, and death reaction \ch{$A$ + $B$ -> $\emptyset$}.
In particular, the input of the A-B protocol are two
  species~$A$ and $B$ with
  an initial total population size~$n = A(0)+B(0)$ and an initial gap~$\Delta=|A(0)-B(0)|$.
We showed that the A-B protocol reaches majority consensus with high probability if the gap weakly grows with the population size according to~$\Delta=\Omega(\sqrt{n \log n})$.
Expected convergence time until consensus is constant and
  in $O(\log n)$ with high probability.

We further demonstrated how to use dual-rail gates to implement digital circuits computing arbitrary Boolean functions.
As opposed to thresholds of a single species, dual-rail
  encoding is particularly useful in our birth systems
  as the A-B protocol allows us to amplify and thus regenerate
  such signals.
  
As a dual-rail gate implementation, we presented the \NAND\ gate protocol that takes two dual-rail encoded input signals and produces a corresponding dual-rail output signal.
The protocol is simple, an important criterion for follow up in real-world biological implementations.
We proved that, given a sufficiently large initial gap between
 the rails of the input signals, our gate produces the correct
 output with high probability in $O(\log n)$ time, where
 $n$ is a lower bound on the initial input population size.
In particular, our gate guarantees an output signal gap of $\Omega(n)$ if both inputs have a gap of at least $0.62$ times their initial population size. By alternating execution of the \NAND\ gate protocol and the A-B protocol, layer by layer, we finally arrive at computing the circuit's outputs. 

Simulations show that the qualitative behavior of our protocols matches the behavior expected from the asymptotic bounds.
While the studied A-B protocol and the \NAND\ gate protocol are simplifications of biological implementations of consensus and gate evaluation protocols,
we believe that our results give a signpost for future research on the successful implementation of complex distributed systems such as indirect inter-cellular communication via phages. 
We discussed a potential biological implementation based on communication by conjugation among engineered \emph{E. coli}.

\begin{acks}
We acknowledge support from the Digicosme working group HicDiesMeus, Ile-de-France (IdF) region's DIM-RFSI, and INRAE's MICA department.
We thank Joel Rybicki for feedback on an earlier version.
\end{acks}

\bibliographystyle{plain}
\bibliography{agents}

\begin{thebibliography}{10}

\bibitem{anderson2006environmentally}
J.~Christopher Anderson, Elizabeth~J. Clarke, Adam~P. Arkin, and Christopher~A.
  Voigt.
\newblock Environmentally controlled invasion of cancer cells by engineered
  bacteria.
\newblock {\em Journal of Molecular Biology}, 355(4):619--627, January 2006.

\bibitem{angluin2008fast}
Dana Angluin, James Aspnes, and David Eisenstat.
\newblock Fast computation by population protocols with a leader.
\newblock {\em Distributed Computing}, 21(3):183--199, 2008.

\bibitem{angluin08:dc}
Dana Angluin, James Aspnes, and David Eisenstat.
\newblock A simple population protocol for fast robust approximate majority.
\newblock {\em Distributed Computing}, 21(2):87--102, March 2008.

\bibitem{billard1974competition}
L~Billard.
\newblock Competition between two species.
\newblock {\em Stochastic Processes and their Applications}, 2(4):391--398,
  1974.

\bibitem{bremaud99}
Pierre Br\'emaud.
\newblock {\em {M}arkov Chains: {G}ibbs Fields, {M}onte {C}arlo Simulation, and
  Queues}.
\newblock Springer, Heidelberg, 1999.

\bibitem{condon19majority}
Anne Condon, Monir Hajiaghayi, David Kirkpatrick, and J\'an~Ma\v nuch.
\newblock Approximate majority analyses using tri-molecular chemical reaction
  networks.
\newblock {\em Natural Computing}, 2019.
\newblock In press.

\bibitem{daniel2013synthetic}
Ramiz Daniel, Jacob~R. Rubens, Rahul Sarpeshkar, and Timothy~K. Lu.
\newblock Synthetic analog computation in living cells.
\newblock {\em Nature}, 497(7451):619--623, May 2013.

\bibitem{DimitriuLBMBLT14}
Tatiana Dimitriu, Chantal Lotton, Julien B{\'e}nard-Capelle, Dusan Misevic,
  Sam~P Brown, Ariel~B Lindner, and Fran{\c{c}}ois Taddei.
\newblock Genetic information transfer promotes cooperation in bacteria.
\newblock {\em Proceedings of the National Academy of Sciences},
  111(30):11103--11108, 2014.

\bibitem{feller39grundlagen}
William Feller.
\newblock Die {G}rundlagen der volterraschen {T}heorie des {K}ampfes ums
  {D}asein in wahrscheinlichkeitstheoretischer {B}ehandlung (1939).
\newblock In {\em Selected Papers I}, pages 441--470. Springer, 2015.

\bibitem{gillespie2001approximate}
Daniel~T Gillespie.
\newblock Approximate accelerated stochastic simulation of chemically reacting
  systems.
\newblock {\em The Journal of chemical physics}, 115(4):1716--1733, 2001.

\bibitem{gomez2012extinction}
Antonio G{\'o}mez-Corral and M~L{\'o}pez Garc{\'\i}a.
\newblock Extinction times and size of the surviving species in a two-species
  competition process.
\newblock {\em Journal of mathematical biology}, 64(1-2):255--289, 2012.

\bibitem{hoeffding1994probability}
Wassily Hoeffding.
\newblock Probability inequalities for sums of bounded random variables.
\newblock In {\em The Collected Works of Wassily Hoeffding}, pages 409--426.
  Springer, 1994.

\bibitem{hoops2006copasi}
S.~Hoops, S.~Sahle, R.~Gauges, C.~Lee, J.~Pahle, N.~Simus, M.~Singhal, L.~Xu,
  P.~Mendes, and U.~Kummer.
\newblock {COPASI}--a {COmplex} {PAthway} {SImulator}.
\newblock {\em Bioinformatics}, 22(24):3067--3074, October 2006.

\bibitem{karlin75}
Samuel Karlin and Howard~M. Taylor.
\newblock {\em A First Course in Stochastic Processes}.
\newblock Academic Press, New York, 2 edition, 1975.

\bibitem{kendall1966branching}
David~G Kendall.
\newblock Branching processes since 1873.
\newblock {\em Journal of the London Mathematical Society}, 1(1):385--406,
  1966.

\bibitem{libby2014ratcheting}
Eric Libby and William~C Ratcliff.
\newblock Ratcheting the evolution of multicellularity.
\newblock {\em Science}, 346(6208):426--427, 2014.

\bibitem{lovasz03}
L.~Lov\'asz, J.~Pelik\'an, and K.~Vesztergombi.
\newblock {\em Discrete Mathematics: Elementary and Beyond}.
\newblock Springer, Heidelberg, 2003.

\bibitem{Lynch96}
Nancy~A Lynch.
\newblock {\em Distributed algorithms}.
\newblock Morgan Kaufmann, 1996.

\bibitem{mahmoud09polya}
Hosam~M. Mahmoud.
\newblock {\em P\'olya Urn Models}.
\newblock CRC Press, Boca Raton, 2009.

\bibitem{marken2019addressable}
John~P Marken and Richard~M Murray.
\newblock Addressable," packet-based" intercellular communication through
  plasmid conjugation.
\newblock {\em bioRxiv}, page 591552, 2019.

\bibitem{moon2012genetic}
Tae~Seok Moon, Chunbo Lou, Alvin Tamsir, Brynne~C. Stanton, and Christopher~A.
  Voigt.
\newblock Genetic programs constructed from layered logic gates in single
  cells.
\newblock {\em Nature}, 491(7423):249--253, October 2012.

\bibitem{myers2001asynchronous}
Chris~J Myers.
\newblock {\em Asynchronous circuit design}.
\newblock John Wiley \& Sons, 2001.

\bibitem{NKK06}
Artem~S Novozhilov, Georgy~P Karev, and Eugene~V Koonin.
\newblock Biological applications of the theory of birth-and-death processes.
\newblock {\em Briefings in bioinformatics}, 7(1):70--85, 2006.

\bibitem{ortiz2012engineered}
Monica~E Ortiz and Drew Endy.
\newblock Engineered cell-cell communication via {DNA} messaging.
\newblock {\em Journal of Biological Engineering}, 6(1):16, December 2012.

\bibitem{paddon2013high}
C.~J. Paddon, P.~J. Westfall, D.~J. Pitera, K.~Benjamin, K.~Fisher, D.~McPhee,
  M.~D. Leavell, A.~Tai, A.~Main, D.~Eng, D.~R. Polichuk, K.~H. Teoh, D.~W.
  Reed, T.~Treynor, J.~Lenihan, H.~Jiang, M.~Fleck, S.~Bajad, G.~Dang,
  D.~Dengrove, D.~Diola, G.~Dorin, K.~W. Ellens, S.~Fickes, J.~Galazzo, S.~P.
  Gaucher, T.~Geistlinger, R.~Henry, M.~Hepp, T.~Horning, T.~Iqbal, L.~Kizer,
  B.~Lieu, D.~Melis, N.~Moss, R.~Regentin, S.~Secrest, H.~Tsuruta, R.~Vazquez,
  L.~F. Westblade, L.~Xu, M.~Yu, Y.~Zhang, L.~Zhao, J.~Lievense, P.~S. Covello,
  J.~D. Keasling, K.~K. Reiling, N.~S. Renninger, and J.~D. Newman.
\newblock High-level semi-synthetic production of the potent antimalarial
  artemisinin.
\newblock {\em Nature}, 496(7446):528--532, April 2013.

\bibitem{ratcliff2012experimental}
William~C Ratcliff, R~Ford Denison, Mark Borrello, and Michael Travisano.
\newblock Experimental evolution of multicellularity.
\newblock {\em Proceedings of the National Academy of Sciences},
  109(5):1595--1600, 2012.

\bibitem{regot2011distributed}
Sergi Regot, Javier Macia, N\'uria Conde, Kentaro Furukawa, Jimmy Kjell\'en,
  Tom Peeters, Stefan Hohmann, Eul\`alia de~Nadal, Francesc Posas, and Ricard
  Sol\'e.
\newblock Distributed biological computation with multicellular engineered
  networks.
\newblock {\em Nature}, 469(7329):207--211, December 2010.

\bibitem{reuter1961competition}
GEH Reuter.
\newblock Competition processes.
\newblock In {\em Proc. 4th Berkeley Symp. Math. Statist. Prob}, volume~2,
  pages 421--430, 1961.

\bibitem{ridler1978competition}
CJ~Ridler-Rowe.
\newblock On competition between two species.
\newblock {\em Journal of Applied Probability}, 15(3):457--465, 1978.

\bibitem{Saaty61}
Thomas~L Saaty.
\newblock {\em Elements of queueing theory: with applications}, volume 34203.
\newblock McGraw-Hill New York, 1961.

\bibitem{schmidl2019rewiring}
Sebastian~R Schmidl, Felix Ekness, Katri Sofjan, Kristina N-M Daeffler,
  Kathryn~R Brink, Brian~P Landry, Karl~P Gerhardt, Nikola Dyulgyarov, Ravi~U
  Sheth, and Jeffrey~J Tabor.
\newblock Rewiring bacterial two-component systems by modular dna-binding
  domain swapping.
\newblock {\em Nature chemical biology}, 15(7):690--698, 2019.

\bibitem{slomovic2015synthetic}
Shimyn Slomovic, Keith Pardee, and James~J Collins.
\newblock Synthetic biology devices for in vitro and in vivo diagnostics.
\newblock {\em Proceedings of the National Academy of Sciences},
  112(47):14429--14435, 2015.

\bibitem{DBLP:journals/nc/SoloveichikCWB08}
David Soloveichik, Matthew Cook, Erik Winfree, and Jehoshua Bruck.
\newblock Computation with finite stochastic chemical reaction networks.
\newblock {\em Natural Computing}, 7(4):615--633, February 2008.

\bibitem{spars2002principles}
Jens Spars and Steve Furber.
\newblock {\em Principles of asynchronous circuit design}.
\newblock Springer, 2002.

\bibitem{tabor2009synthetic}
Jeffrey~J Tabor, Howard~M Salis, Zachary~Booth Simpson, Aaron~A Chevalier,
  Anselm Levskaya, Edward~M Marcotte, Christopher~A Voigt, and Andrew~D
  Ellington.
\newblock A synthetic genetic edge detection program.
\newblock {\em Cell}, 137(7):1272--1281, 2009.

\bibitem{tamsir2011robust}
Alvin Tamsir, Jeffrey~J. Tabor, and Christopher~A. Voigt.
\newblock Robust multicellular computing using genetically encoded {NOR} gates
  and chemical `wires'.
\newblock {\em Nature}, 469(7329):212--215, December 2010.

\bibitem{tay2017synthetic}
Pei Kun~R Tay, Peter~Q Nguyen, and Neel~S Joshi.
\newblock A synthetic circuit for mercury bioremediation using self-assembling
  functional amyloids.
\newblock {\em ACS synthetic biology}, 6(10):1841--1850, 2017.

\bibitem{volterraleccons}
Vito Volterra.
\newblock Le{\c{c}}ons sur la theorie mathematique de la lutte pour la vie.
\newblock {\em Gauthier-Villars, Paris}, 1931.

\end{thebibliography}

\end{document}